\documentclass[final,onefignum,onetabnum]{siamart190516}

\newcommand{\eg}{{\it e.g.}}
\newcommand{\ie}{{\it i.e.}}
\newcommand{\etc}{{\it etc.}}

\newcommand{\BA}{\begin{array}}
\newcommand{\EA}{\end{array}}

\newcommand{\BIT}{\begin{itemize}}
\newcommand{\EIT}{\end{itemize}}
\newcommand{\dom}{\mathop{\bf dom}}

\newcommand{\reals}{{\mathbb{R}}} %

\newcommand{\argmin}{\mathop{\rm argmin}}

\newcommand{\dist}{\mathop{\bf dist}}

\usepackage{color}
\usepackage{url}  %
\usepackage{hyperref}
\usepackage{longtable}

\usepackage{amsfonts}       %
\usepackage{amsmath}
\usepackage{booktabs}
\usepackage{mathtools}

\usepackage{algpseudocode}
\usepackage{caption}
\usepackage{subcaption}%

\newcommand{\Cc}{\mathcal{C}}
\newcommand{\Ac}{\mathcal{A}}
\newcommand{\Fc}{\mathcal{F}}
\newcommand{\Ic}{\mathcal{I}}
\newcommand{\Qc}{\mathcal{Q}}
\newcommand{\Kc}{\mathcal{K}}

\newcommand{\Tc}{\mathcal{T}}

\title{Operator splitting for a homogeneous embedding of the linear
complementarity problem}
\author{Brendan O'Donoghue \\ Deepmind, London, UK}
\begin{document}
\maketitle

\begin{abstract}
We present a first-order quadratic cone programming (QCP) algorithm that can
scale to very large problem sizes and produce modest accuracy solutions quickly.
Our algorithm returns primal-dual optimal solutions when available or
certificates of infeasibility otherwise.  It is derived by applying
Douglas-Rachford splitting to a homogeneous embedding of the linear
complementarity problem,  which is a general set membership problem that
includes QCPs as a special case. Each iteration of our procedure requires
projecting onto a convex cone and
solving a linear system with a fixed coefficient matrix.
If a sequence of related problems are solved then the procedure
can easily be warm-started and make use of factorization caching of the linear
system.  We demonstrate on a range of public and synthetic datasets that for
feasible problems our approach tends to be somewhat faster than applying
operator splitting directly to the QCP, and in cases of infeasibility our
approach can be significantly faster than alternative approaches based on
diverging iterates.  The algorithm we describe has been implemented in C
and is available open-source in the solver SCS v3.0.

\end{abstract}



\section{Introduction}
The goal in a linear complementarity problem (LCP) is to find a point in a
convex cone that satisfies a complementarity condition \cite{cottle1992linear,
eaves1971linear, murty1988linear}. In this paper we apply Douglas-Rachford
splitting to a homogeneous embedding of the monotone LCP that encodes both the
feasibility and infeasibility conditions of the original problem. Although the
algorithm we develop is for general monotone LCPs, in this manuscript we focus
on convex quadratic cone programs (QCP) which are a special case. QCPs are a
type of convex optimization problem where the goal is to minimize a
quadratic objective subject to a conic constraint.

The recent SCS (Splitting Conic Solver) algorithm is a first-order
optimization procedure that can solve large convex \emph{linear} cone problems to
modest accuracy quickly \cite{ocpb:16, scs}. It is based on applying the
alternating directions method of multipliers (ADMM) to a homogeneous self-dual
embedding of the problem \cite{boyd2011distributed,parikh2014proximal,
zheng2019chordal}.  However, it cannot handle quadratic objectives directly,
relying instead on reductions to second-order cone constraints. This reduction
is inefficient in three ways.  First, it is costly to perform the necessary
matrix factorization required for conversion; second, the factorization may
destroy any favourable sparsity in the original data; and third, it appears
that operator splitting methods like ADMM are better able to exploit the strong
convexity of a quadratic objective when used directly, rather than as a
second-order cone \cite{giselsson2016linear, moursi2019douglas}.  This
limitation, and the myriad of real-world applications with quadratic objectives,
has inspired the development of first-order ADMM based solvers that tackle the
quadratic objective directly \cite{o2013splitting, stellato2018osqp,
garstka_2019}.  However, solvers not based on a homogeneous embedding must rely
on an alternative procedure based on diverging iterates to generate
certificates of infeasibility if the problem does not have a solution
\cite{banjac2019infeasibility, liu2017new, hermans2019qpalm,
banjac2021asymptotic, banjac2021, applegate2021infeasibility}. This procedure
tends to be slower and less robust in practice \cite{meszaros2015practical}. In
this paper we derive an algorithm that enjoys both properties - direct handling
of quadratic objectives and efficient generation of infeasibility certificates.

Building on the homogeneous self-dual model of Goldman and Tucker
\cite{goldmantucker} a series of papers developed homogeneous embeddings for the
LCP \cite{ye1994nl, ye1997homogeneous}, and the more general monotone
complementarity problem (MCP) \cite{andersen1999homogeneous}.  Here we use the
embedding of Andersen and Ye \cite{andersen1999homogeneous} applied to a
monotone LCP. We show that the operator corresponding to the embedding is
monotone, but not maximal, a property required for convergence of most operator
splitting techniques. We derive a natural maximal extension of the operator
which defines the final embedding. The resulting embedded problem can be
expressed as finding a zero of the sum of two maximal monotone operators, to
which we can apply standard operator splitting methods \cite{bauschke2011convex,
ryu2016primer}.

We focus our attention on Douglas-Rachford (DR) splitting due to its general
convergence guarantees and good empirical performance
\cite{douglas1956numerical, lions1979splitting}, though there are many
alternative approaches \cite{tseng2000modified, peaceman1955numerical}.  DR
splitting is equivalent to ADMM under a particular change of variables
\cite{gabay1983chapter, eckstein1992douglas} (and indeed both are instantiations
of the proximal point method \cite{rockafellar1976monotone}), and so the final
method we derive is closely related to the SCS algorithm.  Applying DR splitting
to the embedded problem results in an iterative procedure with a per-iteration
cost that is almost identical to the linear-convex case as tackled by SCS
and to applying the splitting method directly to the original
problem \cite{wen2010alternating, o2013splitting}.

There are several advantages that the homogeneous embedding approach has over
competing methods of generating certificates of infeasibility based on
diverging sequences \cite{banjac2019infeasibility, liu2017new}.  When using
the homogeneous embedding infeasibility certificates are generated by
\emph{convergence}. Alternative methods generate certificates by
\emph{divergence}, typically by examining the difference between successive
iterates.  This means when using the homogeneous embedding we have much more
flexibility about how we converge to a solution. For instance, we can apply any
procedure that guarantees convergence to a (nonzero) fixed point, which means we
can use inexact or stochastic updates \cite{rockafellar1976monotone,
eckstein1992douglas}, modern acceleration techniques \cite{goldstein2014fast,
zhang2018globally, themelis2019supermann, sopasakis2019superscs}, or
second-order extensions \cite{ali2017semismooth}. Moreover, approaches relying
on DR splitting automatically benefit from the guaranteed $o(1/k)$ bound on the
convergence rate \cite{he2015convergence, davis2016convergence}.  This is in
contrast to the difference of diverging iterates produced by DR splitting which
have no guaranteed \emph{rate} of convergence in general, satisfying a weaker notion of
convergence instead \cite[Thm.~3]{liu2017new}.  This stronger convergence
guarantee is not just theoretical, since algorithms for the homogeneous
embedding tends to be faster and more robust at detecting infeasibility in
practice. This was shown recently for interior point methods
\cite{meszaros2015practical} and we shall show similar results
experimentally for DR splitting. On the other hand, if the problem is feasible
then using the homogeneous embedding does not appear to harm convergence when
compared to tackling the original problem directly.  On the contrary, we present
numerical evidence to suggest that the homogeneous embedding approach can
actually converge to a solution slightly faster than direct approaches even when
the problem is feasible, at least when DR splitting is used.

QCPs are an important problem type with many applications, some of which we list
here.  Every linear program, quadratic program, second-order cone program,
semidefinite program and exponential cone program, \etc, can be formulated as a
QCP.  Sequential quadratic programming is an effective non-linear constrained
optimization algorithm that relies on solving a sequence of QCPs in order to
converge to the solution of the original non-linear problem
\cite[Ch.~18]{nocedal2006sequential}, \cite{tolle1995sequential}.  In machine
learning the support vector machine and the lasso can both be formulated as QCPs
\cite{noble2006support, tibshirani1996regression}. In portfolio optimization the
standard trade-off between return and risk is often formulated as a QCP once
additional constraints, such as trading costs, leverage limits, \etc, are
incorporated \cite{markowitz1991foundations, boyd2014performance}. Model
predictive control with quadratic stage costs is a QCP with a particular
sparsity structure \cite{camacho2013model, o2013splitting,
stathopoulos2016operator}.  Quadratic objectives over the semidefinite cone come
up when solving matrix reconstruction and low rank matrix completion problems,
where the goal is to find a positive semidefinite matrix with low rank that
minimizes the Frobenius norm to some reference \cite{jain2013low,
koltchinskii2011nuclear}.  Fast and robust generation of certificates of
infeasibility is important in a range of applications. For example, in a branch
and bound procedure applied to a mixed-integer quadratic programming problem
some branches are infeasible and pruning those away reliably is crucial for good
performance \cite{lawler1966branch, nair2020solving}.

\paragraph{Software}
The QCP algorithm we describe in this paper has been
implemented in C and is available online at this URL:
\\
\centerline{\url{https://github.com/cvxgrp/scs/tree/3.0.0}}
It is written as an extension of the SCS solver and it thus inherits
the capabilities of SCS. Specifically, it can solve
convex quadratic cone programs involving any combination of nonnegative,
second-order, semidefinite, exponential, and power cones (and their duals).
It has multi-threaded and single-threaded versions, can run on
both CPU and GPU, and solves the linear system at each iteration using either a
direct method or an iterative method.  It can be used in other C, C++, Python,
MATLAB, R, Julia, and Ruby programs and is a supported solver in parser-solvers
CVX \cite{cvx}, CVXPY \cite{cvxpy}, Convex.jl \cite{convexjl}, JuMP
\cite{DunningHuchetteLubin2017}, and YALMIP \cite{yalmip}.

\section{Monotone operator preliminaries}
\label{s-prelims}
This manuscript is concerned with operator splitting algorithms applied to a
monotone inclusion problem, so here we cover the basic concepts that we use
later; for more detail see, \eg, \cite{bauschke2011convex,
ryu2016primer}.
An operator (or relation, point-to-set mapping, multi-valued function)
$F$ on $\reals^d$ can be characterized by its graph, which
is a subset of $\reals^d \times \reals^d$. We
shall use the notation $F(x)$ to refer to the set $\{y\mid (x,y) \in F\}$.  Many
of the operators we consider in this paper are single-valued, \ie, for a fixed
$x \in \reals^d$ the set $\{ y \mid (x, y) \in F\}$ is a singleton and with some
abuse of notation we shall write $y = F(x)$ in this case.

An operator $F$ is \emph{monotone} if it satisfies
\[
(u - v)^\top(x-y) \geq 0, \mbox{ for all } (x,u), (y,v) \in F,
\]
or in shorthand notation
\[
(F(x) - F(z))^\top(x - z) \geq 0,
\]
for all $x, z \in \dom(F)$, where the domain is taken to be $\dom(F)
= \{x \mid F(x) \neq \emptyset\}$.

A monotone operator is \emph{maximal} if it is not strictly contained by another
monotone operator, \ie, extending $F$ to include $(x, u) \in \reals^d \times
\reals^d$ would result in a non-monotone operator for any $(x, u)$ not already
in $F$. Maximality is not just a technical detail, it is an important property
for convergence of the algorithms we develop in this manuscript and we shall
verify that the operators we present are maximal monotone.  Examples of maximal
monotone operators include the identity operator $I = \{(x,x) \mid x \in
\reals^d \}$ and the subdifferential $\partial f = \{g \mid f(z) \geq f(x) +
g^\top(z-x),~\forall z \in \reals^n\}$ of closed, convex, proper
function $f$ \cite{ryu2016primer}.

\subsection{Operator splitting}
In this manuscript we deal with
monotone inclusion problems involving the sum of two maximal monotone
operators; that is we want to find a $u \in \reals^d$ such that
\begin{equation}
\label{e-inc-prob}
  0 \in F(u) + G(u),
\end{equation}
where $F$ and $G$ are maximal monotone operators on $\reals^d$. Operator
splitting methods are a family of algorithms for finding a zero in this case
whereby we make use of the operators that define the problem separately. In this
manuscript we focus on the well-known Douglas-Rachford splitting
method. DR splitting applied to the inclusion problem \eqref{e-inc-prob} is the
following iterative procedure: From any initial $w^0 \in \reals^d$ repeat
for $k=0,1,\ldots$,
\begin{align}
  \begin{split}
    \label{e-dr}
\tilde u^{k+1} &= (I + F)^{-1} w^k\\
u^{k+1} &= (I + G)^{-1}(2\tilde u^{k+1} - w^k)\\
w^{k+1} &= w^k + u^{k+1} - \tilde u^{k+1}.
  \end{split}
\end{align}
If a solution to \eqref{e-inc-prob} exists, then the DR splitting procedure
generates a sequence of iterates $(w^k, u^k, \tilde u^k)$ that satisfy $\|u^k -
\tilde u^k\| \rightarrow 0$, $u^k \rightarrow u^\star$, and $w^k \rightarrow
w^\star \in u^\star + F(u^\star)$, where $u^\star \in \reals^d$ is a solution
\cite[Thm.~26.11]{bauschke2011convex}. The quantity $\|w^{k+1} -
w^k\|^2_2$ converges to zero at a rate of $o(1/k)$
\cite[Cor.~2]{davis2016convergence}, \cite[Thm.~3.1]{he2015convergence}. If a solution
does not exist then the iterates generated by DR splitting will not converge.

\subsection{Resolvent operator}
The first two steps of DR splitting require the evaluation of the
\emph{resolvent} of the two operators in the inclusion, which for operator $F$
is $(I + F)^{-1}$.  The resolvent
of a maximal monotone operator is always single-valued, even if the operator
that defines it is not, and has full domain \cite{minty1961maximal,
minty1962monotone}. If $F$ is the subdifferential of a convex function $f$,
then the resolvent is known as the \emph{proximal} operator
\cite{parikh2014proximal}, and is given by
\begin{equation}
\label{e-prox}
\begin{array}{lllll}
&&y &=& (I + \partial f)^{-1}x\\
&\Leftrightarrow& 0 &\in& \partial f(y) + y - x \\
&\Leftrightarrow& y &=& \argmin_z \left(f(z) + (1/2)\|z - x\|_2^2 \right).
\end{array}
\end{equation}

\section{The monotone and linear complementarity problems}
\label{s-equivalences}
Quadratic cone programs (QCPs) are the main problems of interest in this paper
and in this section we review the relationship between QCPs and linear
complementarity problems (LCP), which are themselves a special case of monotone
complementarity problems (MCP). We introduce these complementarity problems and
show their equivalence to monotone \emph{inclusion} problems, to which we can
apply operator splitting techniques. In the sequel we shall embed the conditions
for feasibility and infeasibility of an LCP into an MCP.

The \emph{monotone complementarity problem} MCP$(F, \Cc)$ defined by maximal
monotone operator $F$ on $\reals^d$ and nonempty, closed, convex
cone $\Cc$ is to find a point $z \in \reals^d$ for which
\begin{equation}
\label{e-original}
\exists\,w \in F(z) \ \mbox{ s.t. } \  \Cc \ni z \perp w \in \Cc^*,
\end{equation}
where $\Cc^*$ denotes the dual cone to $\Cc$, \ie, $\Cc^* = \{w \mid w^\top z
\geq 0, z \in \Cc\}$.
That is, the problem is to find a
$z \in \Cc$ such that for some $w \in F(z) \cap \Cc^*$ we have $z^\top w = 0$.
If $F$ is single-valued, then we can write the problem more
succinctly as finding a $z \in \reals^d$ such that
$\Cc \ni z \perp F(z) \in \Cc^*$.

Problem \eqref{e-original} is equivalent to finding a $z \in \Cc$ that
satisfies the following variational inequality \cite[Def.~26.19]{bauschke2011convex}
\begin{equation}
\label{e-vi}
\exists\,w \in F(z) \ \mbox{ s.t. } \  (y - z)^\top w \geq 0 \quad \forall y \in \Cc.
\end{equation}
To see this first note that if we have a $(z, w) \in F$ that satisfies \eqref{e-original}
then clearly
\[
y^\top w \geq z^\top w = 0,
\]
for all $y \in \Cc$ since $w \in \Cc^*$. To see the other
direction consider a $(z,w) \in F$ with $z \in \Cc$ that satisfies \eqref{e-vi} and note that if
$z^\top w \neq 0$, then we can take $y = (1/2) z$ or $y = (3/2) z$ to
violate the upper bound property, so it must be the case that $z^\top w = 0$,
then $y^\top w \geq 0$ for all $y \in \Cc$ implies that
$w \in \Cc^*$.

These problems are also equivalent to the problem of finding a $z \in \reals^d$
that satisfies the following inclusion:
\begin{equation}
\label{e-mip}
0 \in F(z) + N_\Cc(z),
\end{equation}
where $N_\Cc(z)$ is the normal cone operator for cone $\Cc$, and is given by
\[
N_\Cc(z) = \left\{
\begin{array}{ll}
\{x \mid (y - z)^\top  x \leq 0, \quad \forall y \in \Cc\} & z \in \Cc\\
\emptyset & z \not \in \Cc.\\
\end{array}
\right.
\]
It is readily shown that $N_\Cc = \partial I_\Cc$, \ie, the subdifferential
of the convex indicator function for $\Cc$. Therefore $N_\Cc$ is maximal
monotone with resolvent $(I + N_\Cc)^{-1}x = \Pi_\Cc(x)$,
the Euclidean projection onto $\Cc$, as can be seen using Equation \eqref{e-prox}.

To see equivalence of problem \eqref{e-vi} and \eqref{e-mip}, note that if $z$
satisfies \eqref{e-mip} then $z \in \Cc$ and there exists $w \in F(z)$ such that
$-w \in N_\Cc(z)$ and so $z$ satisfies \eqref{e-vi} and vice-versa. The sum of
two maximal monotone operators is also maximal monotone, so problem
\eqref{e-mip} is a maximal monotone inclusion problem.

An affine function $F(z) = Mz + q$ with
matrix $M \in \reals^{d \times d}$ and vector $q \in \reals^d$
is maximal monotone if and only if $M$ is monotone, \ie,
\begin{equation} \label{e-mat-monotone}
M + M^\top \succeq 0,
\end{equation}
where we use the notation $\cdot \succeq 0$ to denote membership in the
positive semidefinite cone of matrices. In this case
MCP$(F, \Cc)$ is a monotone \emph{linear complementarity problem} LCP$(M, q, \Cc)$;
\ie, the problem of finding $z \in \reals^d$ such that
\begin{equation}
\label{e-lcp}
\Cc \ni z \perp (Mz + q) \in \Cc^*.
\end{equation}
When $M$ is not monotone then the LCP
may be very difficult to solve \cite{cottle1992linear}. One immediate
consequence of the fact that $M$ is monotone is that
\begin{equation}
\label{e-fact1}
z^\top M z = 0 \ \Leftrightarrow \ (M + M^\top) z = 0,
\end{equation}
which can be seen from the fact that $z^\top M z = (1/2)z^\top(M + M^\top) z =
(1/2) \|(M + M^\top)^{1/2} z\|^2_2$ for any $z \in \reals^d$. We shall make
use of this fact in our analysis.

\subsection{Quadratic cone programming}
As a concrete example of an LCP take the convex
\emph{quadratic cone program} (QCP), which is the following
primal-dual problem pair:
\begin{equation}
\label{e-qcp}
\begin{array}{lr}
\begin{array}{ll}
\mbox{minimize} & (1/2)x^\top  P x + c^\top  x\\
\mbox{subject to} &  Ax + s = b\\
  & s \in \Kc
\end{array}
&
\begin{array}{ll}
\mbox{maximize} & -(1/2)x^\top  P x - b^\top  y\\
\mbox{subject to} &  Px + A^\top y + c = 0\\
  & y \in \Kc^*,
\end{array}
\end{array}
\end{equation}
over variables $x \in \reals^n$, $s \in \reals^m$, $y\in \reals^m$, with data $A
\in \reals^{m \times n}$, $P \in \reals^{n \times n}$, $c \in \reals^n$, $b
\in \reals^m$, where $\Kc$ is a nonempty, closed, convex cone and where $P = P^\top
\succeq 0$ (for a derivation of the dual see \cite[A.2]{banjac2019infeasibility}).
When strong duality holds, the Karush-Kuhn-Tucker (KKT) conditions
are necessary and sufficient for
optimality \cite[\S 5.5.3]{boyd2004convex}. They are given by
\begin{equation}
\label{e-qp-kkt}
  Ax + s = b, \quad Px + A^\top y + c= 0,\quad s \in \Kc, \quad y \in
  \Kc^*,\quad  s \perp y.
\end{equation}
These are primal feasibility, dual feasibility, primal and dual cone membership,
and complementary slackness.  The complementary slackness condition is
equivalent to a zero \emph{duality gap} condition at any optimal point, that is
for $(x,y,s)$ that satisfy the KKT conditions we have
\begin{equation}
\label{e-duality-gap}
s\perp y \ \Leftrightarrow \ c^\top x + b^\top y + x^\top P x = 0.
\end{equation}
The KKT conditions can be rewritten as
\begin{equation}
\label{e-kkt}
  \reals^n \times \Kc^* \ni \begin{bmatrix} x \\ y \end{bmatrix}\perp
    \begin{bmatrix}Px + A^\top  y +c\\b -Ax \end{bmatrix} \in \{0\}^n \times
\Kc,
\end{equation}
which corresponds to LCP$(M,q, \Cc)$ in variable $z \in \reals^d$ with
\begin{equation}
\label{e-qp-lcp}
  z =  \begin{bmatrix} x \\ y \end{bmatrix}, \quad M = \begin{bmatrix}P & A^\top \\ -A & 0
  \end{bmatrix}, \quad q =
    \begin{bmatrix} c \\ b \end{bmatrix}, \quad \Cc = \reals^n \times \Kc^*,
\end{equation}
where dimension $d = n+m$
and $M$ is monotone, \ie, satisfies \eqref{e-mat-monotone}, since $P
\succeq 0$.

If there exists a solution to the QCP, then there exists
a feasible point of the LCP, and vice-versa. If the quadratic cone
program is primal or dual infeasible, then the LCP is infeasible, and
vice-versa.
In this case any $y \in \reals^m$ that satisfies
\begin{equation}
\label{e-qp-infeas}
A^\top y = 0,\  y \in \Kc^*,\ b^\top y < 0
\end{equation}
acts a certificate that the quadratic cone program is primal infeasible (dual
unbounded)
\cite[\S 5.8]{boyd2004convex}.  Similarly, if we can find $x \in \reals^n$ such that
\begin{equation}
\label{e-qp-unbdd}
Px = 0,\ -Ax\in \Kc,\ c^\top x < 0
\end{equation}
then this is a certificate that the problem is dual infeasible (primal
unbounded)
\cite[\S 5.8]{boyd2004convex}.  We shall discuss how these certificates relate to
infeasibility of LCPs in the sequel.

\section{A homogeneous embedding for monotone LCPs}
As we have seen, every monotone LCP can be written as the monotone inclusion
problem in Equation \eqref{e-mip}.  However, if the original LCP is infeasible (when there
does not exist a $z \in \reals^d$ that satisfies the conditions \eqref{e-lcp}) then the
monotone inclusion problem does not have a solution. In this section we derive a
homogeneous embedding that always has a solution, even when the original LCP
is infeasible. To do so we derive two homogeneous MCPs, one that encodes
feasibility and another that encodes (strong) infeasibility.  The final
embedding is then an MCP involving the \emph{union} of these two operators,
which we shall show is maximal monotone.
\subsection{LCP feasibility}
Andersen and Ye developed a homogeneous embedding that encodes the feasibility
conditions for monotone complementarity problems \cite{andersen1999homogeneous}.
When specialized to the $d$-dimensional LCP$(M, q, \Cc)$ case
the (single-valued) embedding operator $\Fc: \reals^d \times
\reals_{++} \rightarrow \reals^{d+1}$ is given by
\begin{equation}
\label{e-Fc}
\Fc(z, \tau) =
\begin{bmatrix}
Mz + q\tau\\
-z^\top Mz/\tau  -z^\top q
\end{bmatrix}
\end{equation}
and the embedded MCP$(\Fc, \Cc_+)$ is to find a $u \in \reals^{d+1}$ such that
\begin{equation}
\label{e-feas-embed}
  \Cc_+ \ni u \perp
    \Fc(u) \in \Cc_+^*,
\end{equation}
where $\Cc_+ = \Cc \times \reals_+$,
with dual cone $\Cc_+^* = \Cc^* \times \reals_+$. Note that
complementarity always holds, since $u^\top \Fc(u) = 0$ for any $u \in
\dom(\Fc)$.  Next we show that MCP$(\Fc, \Cc_+)$ encodes the set of
solutions to LCP$(M, q, \Cc)$.
If there exists a point $z^\star \in \reals^d$ that solves LCP$(M,q,\Cc)$, \ie,
satisfies \eqref{e-lcp}, then for any
$t > 0$
\begin{equation}
  \Cc_+ \ni \begin{bmatrix} tz^\star \\ t \end{bmatrix} \perp
    \begin{bmatrix} t(Mz^\star + q)  \\ 0 \end{bmatrix}  \in \Cc_+^*
\end{equation}
and so $u = (tz^\star, t) \in \reals^d \times \reals_{++}$ is a solution to the homogeneous embedding
\eqref{e-feas-embed}. Now we show the other direction. Let $u = (z,\tau) \in
\dom(\Fc)$, \ie, $\tau > 0$, be a solution to \eqref{e-feas-embed}. We know
that $z^\top (M z + q\tau) = 0$, and so $(z / \tau) \perp (M(z / \tau) + q)$ and
due to the positive homogeneity of cones $z / \tau \in \Cc$ and $(M(z / \tau) +
q) \in \Cc^*$. These imply that the point $z / \tau$ satisfies the conditions of
\eqref{e-lcp}, and so is a solution to LCP$(M,q,\Cc)$.
\begin{lemma}
The operator $\Fc$ is monotone.
\end{lemma}
\begin{proof}
Let $u = (u_z,u_\tau) \in \reals^d \times \reals_{++}$, $w = (w_z, w_\tau) \in
\reals^d \times \reals_{++}$, then, 
\begin{align*}
  (\Fc(u) - \Fc (w))^\top (u - w) &= -(\Fc (u))^\top  w - (\Fc(w))^\top  u \\
  &=-w_z^\top  Mu_z + w_\tau u_z^\top Mu_z / u_\tau
  - u_z^\top
  M w_z + u_\tau w_z^\top M w_z / w_\tau\\
    &= u_\tau w_\tau (Mu_z/u_\tau - M w_z/w_\tau)^\top (u_z/u_\tau-w_z/w_\tau)\\
    &\geq 0,
\end{align*}
since $M$ is monotone and $u_\tau w_\tau > 0$.
\end{proof}
Although $\Fc$ is monotone, it is not \emph{maximal} monotone, which is a
required property for DR splitting to have guaranteed convergence. In order to
extend the operator to be maximal we must consider infeasibility of the original
LCP, which we do next.

\subsection{LCP infeasibility}
Let us denote by $\Ac = \{(z, w) \mid w = -(Mz + q) \}$.
LCP$(M,q,\Cc)$ is feasible if and only if there exists a point $(z,w) \in
N_\Cc \cap \Ac$.  To see this observe that any such point satisfies $-(Mz +
q) = w \in N_\Cc(z)$, so $0 \in (Mz + q) + N_\Cc(z)$, \ie, $z$ satisfies
\eqref{e-mip}.  If $N_\Cc \cap \Ac = \emptyset$ then no such point exists and
the problem is \emph{infeasible}. A stronger condition is that the distance
between the sets $N_\Cc$ and $\Ac$ is strictly
positive, that is
\[
\dist(N_\Cc, \Ac) = \inf_{(z_1, w_1) \in N_\Cc, (z_2, w_2)
\in \Ac}\|(z_1, w_1) - (z_2, w_2)\| > 0,
\]
in which case we refer to the problem
as \emph{strongly infeasible} \cite{luo1997duality, lourencco2016weak,
lourencco2021solving}.  A necessary and sufficient condition for this is that
the sets are \emph{strongly separated} \cite[Ch.~11]{rockafellar1970convex},
which is the existence of a strongly separating hyperplane with
normal vector $(\mu, \lambda)\in\reals^d \times \reals^d$ that satisfies
\[
\inf_{(z,w) \in \Ac } (z^\top \mu + w^\top \lambda) > 0, \quad \sup_{(z,w) \in N_\Cc}
(z^\top \mu + w^\top \lambda) \leq 0,
\]
since $N_\Cc$ is a cone \cite[Thm.~11.7]{rockafellar1970convex}.
We can simplify this by substituting $w = -(Mz + q)$ into the first
condition, yielding
\[
\inf_{z \in \reals^d} (z^\top(\mu - M^\top \lambda) - \lambda^\top q) > 0,
\]
which implies that $\mu = M^\top \lambda$, and consequently
that $\lambda^\top q < 0$. This brings us to necessary and sufficient conditions
for strong infeasibility of LCP$(M,q,\Cc)$, which is the existence of
a $\lambda \in \reals^d$ such that
\begin{equation}
\label{e-strong-infeas}
\lambda^\top q < 0, \quad \sup_{(z,w) \in N_\Cc}
\lambda^\top (M z + w) \leq 0.
\end{equation}
Next we establish that the
above conditions on $\lambda$ can be embedded into another LCP.

\begin{lemma}
\label{l-infeas}
LCP$(M,q, \Cc)$ is strongly infeasible if and only if there exists a $\lambda \in
\reals^d$ with $\lambda^\top q < 0$ that solves LCP$(M, 0, \Cc)$, \ie,
\begin{equation}
\label{e-lcp-infeas}
  \Cc \ni \lambda \perp
   M\lambda \in \Cc^*.
\end{equation}
\end{lemma}
\begin{proof}
First, we show that any certificate of strong infeasibility solves
\eqref{e-lcp-infeas}.
Consider the second condition in \eqref{e-strong-infeas},
setting $z = 0$ yields $w^\top \lambda \leq 0$ for all $w \in N_\Cc(0) =
-\Cc^*$, and so $\lambda \in \Cc$.  For any $z \in \Cc$ we know that $0 \in
N_\Cc(z)$ and so $\lambda^\top M z \leq 0$, which
implies that $-M^\top\lambda \in \Cc^*$. Together these tell us that
$\lambda^\top M \lambda \leq 0$, but since $M$ is monotone it must be that
$\lambda^\top  M \lambda = 0$ and therefore $M\lambda =- M^\top \lambda$, using
\eqref{e-fact1}. Putting it together with the fact that $\lambda^\top q < 0$
yields the final result.

Now we show the other direction, assume $\lambda \in \reals^d$ satisfies
\eqref{e-lcp-infeas} with $\lambda^\top q < 0$.  We must show that this
satisfies the second condition in \eqref{e-strong-infeas}.  Take any $(z, w) \in
N_\Cc$ and $x \in \Cc$, then from the definition of normal cones $x^\top w \leq
z^\top w$. If $x^\top w > 0$ then there must exist some $t > 0$ such that $t
x^\top w > z^\top w$, and since $tx \in \Cc$ this would contradict that fact
that $w\in N_\Cc(z)$. So it must be the case that $x^\top w \leq 0$. Since $x$
was arbitrary in $\Cc$ it implies that $-w \in \Cc^*$, and so $\lambda^\top w
\leq 0$ due to $\lambda \in \Cc$. Since $\lambda^\top M \lambda = 0$ we know
that $M\lambda = -M^\top \lambda \in \Cc^*$ from \eqref{e-fact1}, so $z^\top
(M^\top \lambda) \leq 0$. Summing these two yields $\lambda^\top(Mz + w) \leq
0$ for any $(z,w) \in N_\Cc$.
\end{proof}

We call any $\lambda$ that satisfies \eqref{e-strong-infeas} a proof or
certificate of (strong) infeasibility. The existence of such a $\lambda$
precludes the existence of $(z, w) \in N_\Cc \cap \Ac$, and any $(z, w)
\in N_\Cc \cap \Ac$ acts as a certificate that there is no $\lambda$ satisfying
\eqref{e-strong-infeas}. In other words at most one of \eqref{e-strong-infeas}
and \eqref{e-lcp} has a solution and they are therefore \emph{weak
alternatives}. This can also been proven directly from the LCPs:
Assume that we have found both a $z \in
\reals^d$ that solves LCP$(M,q,\Cc)$ and a $\lambda \in \reals^d$ that solves
LCP$(M,0,\Cc)$ with $\lambda^\top q < 0$. Then $z + \lambda
\in \Cc$ and $M(z + \lambda) + q \in \Cc^*$ and from cone duality
$0 \leq (z+\lambda)^\top (M(z + \lambda) + q) = \lambda^\top q < 0$, which is a
contradiction.

In the special case of a QCP satisfying strong duality then \emph{exactly} one
of those two systems has a solution and they are \emph{strong alternatives}
\cite[\S 5.8]{boyd2004convex}.

\subsubsection{QCP infeasibility}
Here we show that the conditions in Equation \eqref{e-lcp-infeas} are exactly
equivalent to the conditions of (strong) primal infeasibility
\eqref{e-qp-infeas} or (strong) dual infeasibility \eqref{e-qp-unbdd} in the
case where we are solving a QCP, and that any certificate for one can be
converted into a certificate for the other.

First, consider the case where $y \in \reals^m$ is a certificate of
primal infeasibility for the QCP, then $\lambda = (0 , y) \in
\reals^n\times\reals^m$ is a certificate for the LCP
since it is readily verified to satisfy the conditions in \eqref{e-lcp-infeas}
with $\lambda^\top q = b^\top y < 0$.
Similarly, if $x \in \reals^n$ is a certificate of dual infeasibility for the QCP,
then $\lambda = (x, 0) \in \reals^n\times\reals^m$ is a certificate of
infeasibility for the LCP by the same logic.

Now consider $\lambda = (x,y) \in \reals^n\times\reals^m$ a certificate of
infeasibility for LCP$(M,q,\Cc)$ corresponding to a QCP, in
which case using Equation \eqref{e-lcp-infeas} we have
\begin{equation}
  \reals^n \times \Kc^* \ni \begin{bmatrix} x \\ y \end{bmatrix} \perp
  \begin{bmatrix} P x + A^\top y \\ -Ax\end{bmatrix}  \in \{0\}^n \times \Kc.
\end{equation}
First note that $y \in \Kc^*$ and $-Ax \in \Kc$.  The second orthogonality
condition implies that $y^\top A x = 0$. From this and the first orthogonality
condition we can infer that $x^\top P x = 0$ and so $Px = 0$, and therefore
$A^\top y = 0$ due to the $\{0\}^n$ cone membership.  Finally, $q^\top \lambda =
c^\top x + b^\top y < 0$ by assumption, and so at least one of $c^\top x$ or
$b^\top y$ is negative. If $c^\top x <0$, then $x$ is a certificate for
dual infeasibility for the QCP since it satisfies \eqref{e-qp-unbdd}, on the other
hand if $b^\top y < 0$ then $y$ is a certificate of primal infeasibility since it
satisfies \eqref{e-qp-infeas}. If both $c^\top x$ and $b^\top y$ are negative
then the original problem is both primal and dual infeasible.

\subsection{Infeasibility embedding}
Here we introduce a homogeneous operator that encodes the
infeasibility conditions for LCP$(M,q,\Cc)$. It will become clear
why we need this operator in the next section when we use it to derive the
complete embedding.
Based on lemma \ref{l-infeas} we define the operator $\Ic$ on $\reals^{d+1}$ as
\begin{equation}
\label{e-Ic}
\Ic(z,\tau) =
\left\{\begin{bmatrix} Mz \\ \kappa \end{bmatrix} \Biggm\vert \kappa \leq
-z^\top q\right\}, \quad \dom(\Ic) = \{(z, 0) \mid z^\top Mz = 0\}
\end{equation}
where $(z,\tau) \in \reals^d \times \reals$.
Consider MCP$(\Ic, \Cc_+)$, that is the problem of finding $u \in \reals^{d+1}$ for which
\begin{equation}
\label{e-infeas}
\exists\,v \in \Ic(u) \ \mbox{ s.t. } \  \Cc_+ \ni u \perp v \in \Cc_+^*.
\end{equation}
Note that again complementarity is always satisfied, \ie, $u^\top v = 0$
for all $(u, v) \in \Ic$.
If $\lambda$ is a certificate of infeasibility for LCP$(M, q, \Cc)$
then $(\lambda, 0) \in \dom(\Ic)$ and
\[
\begin{bmatrix}  M \lambda \\- \lambda^\top q\end{bmatrix} \in \Ic(\lambda,0),
\]
and therefore $(\lambda, 0)$ is a solution to MCP$(\Ic, \Cc_+)$. On the other
hand, any solution $u$ to MCP$(\Ic, \Cc_+)$ such that $(w, \kappa) = v \in \Ic(u)$
with $\kappa > 0$ yields a certificate of infeasibility for
LCP$(M, q, \Cc)$.
\begin{lemma}
The operator $\Ic$ is monotone.
\end{lemma}
\begin{proof}
Let $u = (u_z,0) \in \reals^d \times \reals$ and $w = (w_z,0) \in \reals^d
\times \reals$ such that $u, w \in \dom(\Ic)$, then,
\begin{align*}
(\Ic(u) - \Ic(w))^\top (u - w) &= -\Ic(u)^\top w - \Ic(w)^\top u\\
&=-u_z^\top M w_z - w_z^\top M u_z\\
&=-w_z^\top(M + M^\top)u_z\\
&=0,
\end{align*}
since $(M + M^\top)u_z = (M + M^\top)w_z = 0$ using Equation \eqref{e-fact1}.
\end{proof}

\subsection{Final embedding}
We have two homogeneous monotone operators, $\Fc$ and $\Ic$, with associated
problems MCP$(\Fc, \Cc_+)$ and MCP$(\Ic, \Cc_+)$ that encode feasibility and
infeasibility of the original problem LCP$(M, q, \Cc)$ respectively.  However, neither
of these operators are maximal. Here we show that the \emph{union} of the two
operators is maximal monotone, and the associated MCP encodes both feasibility
and infeasibility of the original LCP. Let
\[
\Qc = \Fc \cup \Ic,
\]
with $\dom(\Qc) = \dom(\Fc) \cup \dom(\Ic)$. The operator $\Qc$
satisfies complementarity, \ie, $u^\top v = 0$ for all $(u, v) \in \Qc$, and is
positively homogeneous, \ie, $\Qc(tu) = t \Qc(u)$ for any $t > 0$. We shall show
that $\Qc$ is \emph{maximal} monotone in the sequel.
The embedded problem is to solve
MCP$(\Qc, \Cc_+)$, \ie, find a $u \in \reals^{d+1}$ for which
\begin{equation}
\label{e-homog-embed}
\exists\,v \in \Qc(u) \ \mbox{ s.t. } \  \Cc_+ \ni u \perp v \in \Cc_+^*,
\end{equation}
which from \S \ref{s-equivalences} we know is equivalent to
the monotone inclusion
\begin{equation}
\label{e-homog-monoinc}
0 \in \Qc(u) + N_{\Cc_+}(u).
\end{equation}
Since both $\Qc$ and $N_{\Cc_+}$
are maximal monotone we can apply operator splitting
methods to solve this problem, which we do in the next section. First, we discuss
how the solutions to MCP$(\Qc, \Cc_+)$ encode the solutions or certificates
of infeasibility to LCP$(M, q, \Cc)$.
Let $u^\star = (z^\star, \tau^\star) \in \reals^d \times \reals$ be any point
that satisfies Equation \eqref{e-homog-embed},
and let $(w^\star, \kappa^\star) = v^\star \in \Qc(u^\star)$. From complementarity
we know that
\[
(u^\star)^\top v^\star = (z^\star)^\top w^\star+ \tau^\star \kappa^\star = 0.
\]
However, $(z^\star)^\top w^\star \geq 0$ and $\tau^\star \kappa^\star \geq 0$ since
$\Cc_+$ and $\Cc_+^*$ are dual, and so it must be that $z^\star \perp w^\star$
and at most one of $\tau^\star$ and $\kappa^\star$ can be positive. When
$\tau^\star
> 0$ then $\kappa^\star =0$, $u^\star \in \dom(\Fc)$, $v^\star = \Fc(u^\star)$, the problem is
feasible and a solution to
LCP$(M, q, \Cc)$ can be derived from $u^\star$.  When $\kappa^\star > 0$ then
$\tau^\star = 0$, $u^\star \in \dom(\Ic)$, $v^\star \in \Ic(u^\star)$, the
problem is infeasible and a certificate of infeasibility of LCP$(M, q, \Cc)$ can
be obtained from $u^\star$.  The next case to consider is when $\tau^\star =
\kappa^\star =
0$, with $u \neq 0$. This is pathological and rarely arises
in practice \cite{ye1997homogeneous}. We can rule out some situations for this
case; for example, if the set of solutions to the LCP is non-empty and bounded
then this pathology cannot occur.
On the other hand, if the LCP is weakly infeasible then the only solutions to
the homogeneous embedding have this form. This includes, for example, feasible
QCPs that do not satisfy strong duality.  However, in
that case it may be possible to modify the problem using facial reduction
techniques \cite{permenter2017solving} or to understand the pathology by
examining how the iterates behave \cite{liu2017new}.

These cases are summarized in Table~\ref{t-cases}. The only other possibility
we must consider is the trivial solution $u = 0$, which is always a solution to
MCP$(\Qc, \Cc_+)$, no matter the problem data. However, we shall prove later that
DR splitting will not converge to
zero if properly initialized, so we can safely ignore this possibility.
\begin{table}
\begin{center}
\begin{tabular}{c|cc}
 & $\tau^\star > 0$ & $\tau^\star = 0$\\
 \hline
$\kappa^\star > 0$ & N/A& Infeasible \\
$\kappa^\star = 0$ & Solved & Pathological.
\end{tabular}
\caption{How the solutions of the MCP relate to the status of the LCP.}
\label{t-cases}
\end{center}
\end{table}

\subsection{Maximal monotonicity of $\Qc$}
In order to apply DR splitting to problem \eqref{e-homog-monoinc} we need $\Qc$
to be \emph{maximal monotone}, without which convergence is not guaranteed.
\begin{lemma}
\label{l-q-maximal}
The operator $\Qc = \Fc \cup \Ic$ is maximal monotone.
\end{lemma}
\begin{proof}
Since $\Fc$ and $\Ic$ are both monotone, to show that $\Qc$ is monotone
we need only consider points $u \in \dom(\Fc)$ and $w \in \dom(\Ic)$.
Let $u = (u_z, u_\tau) \in \reals^d \times \reals_{++}$, $w = (w_z, 0) \in
\reals^d \times \reals$, and $(Mw_z, \kappa) \in \Ic(w)$, then
\begin{align*}
(\Qc(u) - \Qc(w))^\top(u - w) & = -\Qc(u)^\top w - \Qc(w)^\top u\\
&=-\Fc(u)^\top w - \Ic(w)^\top u\\
&\ni-w_z^\top(M u_z + q u_\tau) - u_z^\top(M w_z) - u_\tau \kappa\\
&=-u_z^\top(M + M^\top)w_z - u_\tau(\kappa + w_z^\top q)\\
&\geq 0,
\end{align*}
since $(M + M^\top)w_z = 0$ and $\kappa \leq -w_z^\top q$. Since it holds for
any $\kappa$ this establishes that $\Qc$ is monotone; next we shall show
maximality.

For any monotone operator there exists a
maximal monotone extension of it with domain contained in the closure of the
convex hull of its domain \cite[Thm.~21.9]{bauschke2011convex}.  The
domain of $\Fc$ is $\reals^d \times \reals_{++}$ which is convex, and so there
exists a maximal monotone extension of $\Fc$ with domain contained in $\reals^d
\times \reals_+$. Let $\overline{\Fc}$ denote such an extension.
We shall show that $\overline{\Fc}$ is unique and $\overline{\Fc} = \Qc$.

To construct the extension we need to find all pairs $(p, r)$ such that $\Fc
\cup \{p, r \}$ is monotone, with $p \in \reals^d \times \reals_+$.  Since $\Fc$ is
continuous on the interior of its domain we can use standard arguments to show
that no such extension pair with $p \in \dom(\Fc)$ exists \cite{bauschke2011convex}. So any
extension pairs $(p, r)$ must have $p$ on the boundary of $\reals^d \times
\reals_+$, which, if we let $p = (p_z, p_\tau)\in \reals^d \times \reals$, corresponds to points
with $p_\tau = 0$.  Let $u = (z,\tau) \in \dom(\Fc)$ and
consider points $p = (p_z, 0) \in \reals^d \times \reals$
and $r = (r_z, r_\tau) \in \reals^d \times \reals$. The monotone property
implies that $(p,r)$ must satisfy
\begin{align*}
0 &\leq (\Fc(u) - r)^\top(u - p) \\
&=-\Fc(u)^\top p - r^\top(u - p)\\
&=-p_z^\top M z - \tau p_z^\top q -r_z^\top (z - p_z) - r_\tau \tau.
\end{align*}
Since $z$ is arbitrary this implies that
$M^\top p_z + r_z = 0$, which in turn implies that
\[
  0 \leq -\tau (p_z^\top q + r_\tau)  -p_z^\top M p_z.
\]
Letting $\tau \rightarrow 0$
we get $p_z^\top M p_z \leq 0$,
but since $M$ is monotone this implies that
\begin{equation}
\label{e-max1}
p_z^\top M p_z = 0
\end{equation}
and so $Mp_z = -M^\top p_z$ from \eqref{e-fact1}, which yields
\begin{equation}
\label{e-max2}
Mp_z = r_z.
\end{equation}
Finally, since $\tau \geq 0$ we have
\begin{equation}
\label{e-max3}
r_\tau \leq -p_z^\top q.
\end{equation}
The conditions \eqref{e-max1}, \eqref{e-max2}, \eqref{e-max3} on $(p, r)$ are
exactly the conditions for $(p, r) \in \Ic$, from the definition of $\Ic$ in
Equation \eqref{e-Ic}. Thus all extension pairs must be elements of $\Ic$ and
so $\overline{\Fc} \subseteq \Fc \cup \Ic = \Qc$. However, it cannot be the case
that $\overline{\Fc} \subset \Qc$ strictly, as $\Qc$ is monotone that would
violate maximality of $\overline{\Fc}$.  Therefore we can conclude that
$\overline{\Fc} = \Qc$, \ie, $\Qc = \Fc \cup \Ic$ is a maximal monotone
extension of $\Fc$.
\end{proof}

\section{Douglas-Rachford splitting for LCPs}
\label{s-dr-for-lcp}
We have discussed how the feasibility and infeasibility conditions for an LCP can
be embedded into a single homogeneous MCP. In this section we apply DR splitting
to MCP$(\Qc, \Cc_+)$, the algorithm that solves
the homogeneous embedded problem is the
main result of this manuscript.

We have established that the operator $\Qc$ is maximal monotone (as is $N_\Cc$).
This implies that DR splitting applied to MCP$(\Qc, \Cc_+)$
will enjoy the convergence properties discussed in \S \ref{s-prelims}.
That is from any initial $w^0 \in \reals^{d+1}$ the procedure
for $k=0,1,\ldots$,
\begin{align}
  \begin{split}
    \label{e-drmcp}
\tilde u^{k+1} &= (I + \Qc)^{-1} w^k\\
u^{k+1} &= \Pi_{\Cc_+}(2\tilde u^{k+1} - w^k)\\
w^{k+1} &= w^k + u^{k+1} - \tilde u^{k+1},
  \end{split}
\end{align}
will converge to a fixed point from which we can derive a solution or a
certificate of infeasibility for the original LCP$(M, q, \Cc)$.
The remaining difficulty is the evaluation of the resolvent of $\Qc$,
which we discuss in the sequel.

By way of comparison, we can also apply DR splitting
to LCP$(M,q,\Cc)$ directly, which yields the following procedure; from any
initial $w^0 \in \reals^d$ for $k=0,1,\ldots$,
\begin{align}
  \begin{split}
    \label{e-drlcp}
\tilde u^{k+1} &= (I + M)^{-1} (w^k - q)\\
u^{k+1} &= \Pi_{\Cc}(2\tilde u^{k+1} - w^k)\\
w^{k+1} &= w^k + u^{k+1} - \tilde u^{k+1}.
  \end{split}
\end{align}
If a solution to LCP$(M,q,\Cc)$ exists then this procedure will converge,
otherwise it has no fixed point and will not converge.

\subsection{Evaluating the resolvent of $\Qc$}
Since $\Qc$ is maximal
monotone we know that the resolvent is single-valued and has full domain
\cite{ryu2016primer}. At time-step $k$ of DR splitting we must solve a system of equations
involving the resolvent of $\Qc$, that is solve
\[
   \begin{bmatrix} z \\ \tau \end{bmatrix} =
     (I + \Qc)^{-1} \begin{bmatrix} \mu^k \\ \eta^k \end{bmatrix}
\]
for a fixed right-hand side $(\mu^k, \eta^k) \in \reals^d \times\reals$.
Suppose for a moment we know that $(z,\tau) \in \dom(\Fc)$, \ie, $\tau > 0$, then using Equation
\eqref{e-Fc} we must solve
\begin{align}
\begin{split}
\label{e-resolvQ}
  (I + M)z + q \tau &= \mu^k\\
  \tau^2 -  \tau(\eta^k + z^\top q) - z^\top M z &= 0,
\end{split}
\end{align}
for $z \in \reals^d$ and $\tau > 0$.  Since $M$ is monotone we have
$z^\top M z \geq 0$, and so one root of the quadratic equation is
nonnegative and one is nonpositive, and since $(z, \tau) \in \dom{\Qc}$ it is the
nonnegative root that corresponds to the solution.
The solution to these
equations also encodes the solution when $(z,\tau) \in \dom(\Ic)$, since if
$z^\top M z = 0$ then the nonnegative root is given by $\tau = \max(0, \eta^k +
z^\top q)$.  In other words, $\tau = 0$ if and only if $\eta^k \leq- z^\top q$
and $z^\top M z = 0$, which are the conditions for $(z,\tau) \in \dom(\Ic)$ in
Equation \eqref{e-Ic}.  This means the solution of \eqref{e-resolvQ} for $\tau
\geq 0$ yields the resolvent of $\Qc$ for any right-hand side.
Let us denote by
\[
p^k = (I + M)^{-1} \mu^k,\quad r = (I + M)^{-1} q,
\]
then we have
\[
  z = p^k - r \tau
\]
for unknown $\tau \geq 0$, and note that since $r$ is constant for all iterations we
only need to compute it once at the start
of the procedure and then reuse this cached value thereafter. To solve for $\tau$
we substitute $z = p^k - r\tau$ into the quadratic Equation \eqref{e-resolvQ}
yielding
\begin{align}
\begin{split}
\label{e-quad-tau}
  0 &= \tau^2 -  \tau(\eta^k + z^\top q) - z^\top ((I + M) z - z)\\
  &= \tau^2 -  \tau(\eta^k + (p^k - r \tau)^\top q) - (p^k - r \tau)^\top
  (\mu^k -q \tau - p^k + r \tau)\\
  &=\tau^2(1 + r^\top r) + \tau (r^\top \mu^k - 2 r^\top p^k- \eta^k) +
  (p^k)^\top (p^k -\mu^k),
\end{split}
\end{align}
and for brevity we denote $\verb|root|_+(\mu^k, \eta^k, p^k, r)$ to be the
nonnegative root of the quadratic Equation \eqref{e-quad-tau} when evaluated
with input values $(\mu^k, \eta^k, p^k, r)$.
Specifically, let
$a = 1 + r^\top r$, $b^k = r^\top \mu^k - 2 r^\top p^k- \eta^k$, and $c^k = (p^k)^\top
(p^k -\mu^k)$, then
\begin{equation}
\label{e-rootplus}
  \verb|root|_+(\mu^k, \eta^k, p^k, r) = \left(-b^k + \sqrt{(b^k)^2 - 4ac^k}\right) / 2a.
\end{equation}
Since the resolvent has full domain it always has a real-valued solution
for any input, which implies that the above quadratic equation always has real
roots. This fact can also be seen directly from the equations by noting that
$(b^k)^2 \geq 0$, $a = 1 + r^\top
r \geq 0$ and $c^k = (p^k)^\top (p^k -\mu^k) = -(p^k)^\top M p^k  \leq 0$
since $M$ is monotone, and so $(b^k)^2 -4ac^k \geq 0$.

\subsection{Final algorithm}
With the resolvent of $\Qc$ in place we are ready to present DR splitting
applied to problem \eqref{e-homog-monoinc} as Algorithm~\ref{a-drmcp}.  The
$u^k, \tilde u^k, w^k$ terms in Algorithm~\ref{a-drmcp} are simply to relate the
procedure to that described in Equation \eqref{e-drmcp}.

Note that neither Algorithm \ref{a-drmcp} nor the procedure described in
Equation \eqref{e-drlcp} has any explicit hyper-parameters (\eg, step-size,
\etc), though in practice the relative scaling of the problem data can have a
large impact on the convergence of the algorithm and most practical solvers
based on DR splitting or ADMM implement some sort of heuristic data rescaling
\cite{ocpb:16, fougner2018parameter, chu2013primal, giselsson2016linear}.

Algorithm \ref{a-drmcp} and the procedure in Equation \eqref{e-drlcp} differ
only in that Algorithm \ref{a-drmcp} maintains an additional set of scalar
parameters ($\tau$, $\tilde \tau$, and $\eta$), and consequently the
computational costs of the two algorithms are essentially the same.  However,
Equation \eqref{e-drlcp} will not converge if the LCP is infeasible, whereas
Algorithm \ref{a-drmcp} will always converge and will produce a certificate of
infeasibility should one exist.  In fact, the procedure in Equation
\eqref{e-drlcp} can be interpreted as Algorithm \ref{a-drmcp} where we fix the
scalar parameters $\tau = \tilde \tau = \eta = 1$. It may be the case that this
is not the best choice for any particular problem and allowing these scale
parameters to vary makes the problem easier, even for feasible cases. We
shall present some preliminary evidence of this effect in the numerical
experiments sections.

For the special case of QCPs with $P=0$ the problem reduces to a linear cone
program of the form that the original SCS algorithm \cite{ocpb:16} was developed
to tackle. Unsurprisingly, we recover SCS from Algorithm~\ref{a-drmcp}
in this case (modulo the change of variables required to go from ADMM to DR
splitting), with the minor difference that Algorithm~\ref{a-drmcp} constrains
the $\tilde \tau^k$ variable to always be nonnegative which is not the case
in SCS.

\begin{algorithm}[t]
\caption{DR splitting for the homogeneous embedding of LCPs}
\begin{algorithmic}
\State {\bf Input:} LCP$(M, q, \Cc)$
\State compute $r = (I+M)^{-1} q$
\State initialize $\mu^0 \in \reals^d$, $\eta^0 > 0$
\For{$k=0,1,\ldots$}
\[
\begin{array}{rcl}
\tilde u^{k+1} &:& \left\{
\begin{array}{rcl}
  p^k &=& (I+M)^{-1} \mu^k\\
  \tilde\tau^{k+1} &=& \verb|root|_+(\mu^k, \eta^k, p^k, r)\\
  \tilde z^{k+1} &=& p^k - r\tilde \tau^{k+1}
\end{array}
\right.\\
\\
u^{k+1} &:& \left\{
\begin{array}{rcl}
  z^{k+1} &=& \Pi_{\Cc}(2 \tilde z^{k+1} - \mu^k)\\
  \tau^{k+1} &=& \Pi_{\reals_+}(2 \tilde \tau^{k+1} - \eta^k)
\end{array}
\right.\\
\\
w^{k+1} &:& \left\{
\begin{array}{rcl}
  \mu^{k+1} &=& \mu^k + z^{k+1} - \tilde z^{k+1}\\
  \eta^{k+1} &=& \eta^k + \tau^{k+1} - \tilde\tau^{k+1}
\end{array}
\right.
\end{array}
\]
\EndFor
\end{algorithmic}
\label{a-drmcp}
\end{algorithm}

\subsection{Eliminating the trivial solution}
Since problem \eqref{e-homog-monoinc} is homogeneous the point $u=0$ is a solution
no matter the data, and we might worry that our approach will converge to zero,
or to a point so close to zero that it is impossible to recover a solution to
the original LCP in a numerically stable way.  Here we generalize a result from
\cite{ocpb:16} to show that this cannot happen so long as the procedure is
initialized correctly.
\begin{lemma}
\label{l-eliminate-trivial}
Fix $w^0 \in \reals^d$ and consider the sequence $w^{k+1} = \Tc(w^k)$ for
$k=0,1, \ldots$, generated by $\Tc : \reals^{d} \rightarrow \reals^{d}$. If
\begin{enumerate}
\item $\Tc$ is positively homogeneous, \ie, $\Tc(tv) = t\Tc(v)$ for any $t > 0$,
$v \in \reals^d$,
\item $\Tc$ has a non-zero fixed point $w^\star \in \reals^p$
which satisfies $(w^\star)^\top w^0 > 0$,
\item $\Tc$ is non-expansive toward any fixed point, \ie, $\|\Tc(v) -
w^\star\|_2 \leq \|v - w^\star\|_2$ for any $v \in \reals^d$,
\end{enumerate}
then for all $k$,
\[
  \|w^k\|_2 \geq \frac{(w^\star)^\top  w^0}{\|w^\star\|_2} > 0.
\]
\end{lemma}
\begin{proof}
Since $\Tc$ is positively homogeneous the point $t w^\star$ is also a fixed
point for any $t > 0$, and since $\Tc$ is non-expansive toward any fixed point
we have
\[
\begin{array}{lllll}
  &&\|w^k - t w^\star\|_2^2 &\leq&  \|w^0 - t w^\star\|_2^2\\
  &\Rightarrow&  - 2 t (w^\star)^\top w^k &\leq&  \|w^0\|_2^2 - 2 t (w^\star)^\top w^0\\
  &\Rightarrow&  \|w^\star\|_2 \|w^k\|_2 &\geq& (w^\star)^\top w^0- \|w^0\|_2^2
/ 2t,
\end{array}
\]
where we used Cauchy-Schwarz in the last line, and letting $t \rightarrow
\infty$ yields the desired result.  \end{proof} If the operator $\Tc$
corresponds to one step of DR splitting then it is \emph{globally} non-expansive
\cite{bauschke2011convex}. When applied to MCP$(\Qc, \Cc_+)$ DR splitting is
positively homogeneous, since both $\Qc$ and $N_\Cc$ are positively homogeneous.
Finally, if we assume that either an optimal solution or a certificate of
infeasibility exists for LCP$(M, q, \Cc)$ then it has a non-zero fixed point,
and since $w^\star \in u^\star + \Qc(u^\star)$ \cite{bauschke2011convex}, where
$u^\star$ is a solution to MCP$(\Qc, \Cc_+)$, it is easy to initialize in such a
way that the condition $(w^\star)^\top w^0 > 0$ is satisfied. For example, we
can set the last entry of $w^0$ to one, and the rest of the entries zero.
Therefore under normal conditions DR splitting satisfies the conditions of the
lemma and so Algorithm \ref{a-drmcp} will converge to a point that is bounded
away from zero.

\subsection{Convergence of Algorithm \ref{a-drmcp}}
\label{s-convergence}
The convergence guarantees for DR splitting tell us that $u^k \rightarrow
u^\star$, $w^k \rightarrow w^\star \in u^\star + \Qc(u^\star)$ and $\|u^k -
\tilde u^k\| \rightarrow 0$, where $u^\star$ is a solution to MCP$(\Qc, \Cc_+)$
\cite[Thm.~26.11]{bauschke2011convex}.
A solution always exists since $u^\star = 0$ is a solution, though we know from
Lemma~\ref{l-eliminate-trivial} that the procedure will not converge to zero
under benign conditions.

Consider the sequence defined as $v^{k+1} =
u^{k+1} + w^{k} - 2 \tilde u^{k+1}$ for $k=0,1,\ldots$. This sequence converges to $\Qc(u^\star)$ since
\begin{align*}
v^{k+1} = u^{k+1} + w^k - 2 \tilde u^{k+1} \rightarrow w^\star - u^\star \in \Qc(u^\star).
\end{align*}
Furthermore, substituting in for $u^{k+1}$ from Equation \eqref{e-drmcp}
combined with the Moreau decomposition \cite{parikh2014proximal,o2019hamiltonian} yields
\begin{align*}
v^{k+1} &= u^{k+1} + w^k - 2 \tilde u^{k+1}\\
&=\Pi_{\Cc_+}(2 \tilde u^{k+1} - w^k) + w^k - 2 \tilde u^{k+1}\\
&=\Pi_{\Cc_+^*}(-2 \tilde u^{k+1} + w^k).
\end{align*}
That is, $u^{k+1}$ and $v^{k+1}$ correspond to the orthogonal Moreau decomposition of $2 \tilde u^{k+1}
- w^k$ onto the cone $\Cc_+$ and its polar (negative dual) cone, which implies
that $v^k \in \Cc_+^*$ and $u^k \perp v^k$ for all $k$.
In summary, the iterates $(u^k, v^k)$ satisfy
\begin{equation}
\label{e-uv-lcp}
\Cc_+ \ni u^k \perp v^k \in \Cc_+^*,
\end{equation}
for all $k$, and the condition $v^k \in \Qc(u^k)$ holds in the limit, \ie,
the pair $(u^k, v^k)$ eventually satisfy the conditions in Equation
\eqref{e-homog-embed}.
Now take the special case of a QCP where $u^k = (x^k, y^k, \tau^k) \in \reals^n
\times \reals^m \times \reals$ and $v^k = (0, s^k, \kappa^k) \in \{0\}^n \times
\reals^m \times \reals$. If $\tau^k \rightarrow \tau^\star > 0$ then, since
$v^k$ converges to $\Qc(u^k)$, these iterates will in the limit provide a
solution which satisfies the KKT conditions \eqref{e-kkt}, \ie, $(x^k
/ \tau^k, s^k/\tau^k, y^k/\tau^k) \rightarrow (x^\star, s^\star, y^\star)$.  Due
to Equation \eqref{e-uv-lcp} we know that $s^k/\tau^k \in \Kc$, $y^k/\tau^k \in
\Kc^*$, and $s^k/\tau^k \perp y^k /\tau^k$ for all $k$ so three of the KKT
conditions are always satisfied by this sequence.  Therefore to check for
optimality we only need to test that the primal residual, dual residual, and the
duality gap defined in Equation \eqref{e-duality-gap} are less than some
tolerance. On the other hand if $\kappa^k \rightarrow \kappa^\star > 0$ then the
iterates will converge to a certificate of primal infeasibility
\eqref{e-qp-infeas} or dual infeasibility \eqref{e-qp-unbdd}. To check for
infeasibility we only need to check that the certificate residuals are below
some tolerance and that either $c^\top x^k < 0$ or $b^\top y^k < 0$, since both
$s^k$ and $y^k$ satisfy the cone membership requirement.

\section{Implementation details for QCPs}
The algorithm we have derived applies to any monotone LCP. In this section we
discuss how to perform the steps in Algorithm \ref{a-drmcp} efficiently for the
QCP special case.

\subsection{Solving the linear system}
In both the procedure described in Equation \eqref{e-drlcp} and Algorithm
\ref{a-drmcp} we need to solve a system of equations with the same matrix at
every iteration. For the specific case of a QCP the
linear system can be written
\[
\begin{bmatrix} I + P & A^T  \\ A & -I \end{bmatrix}\begin{bmatrix} x \\ y
\end{bmatrix} = \begin{bmatrix} \mu_x \\ -\mu_y \end{bmatrix},
\]
for $(x,y) \in \reals^n \times \reals^m$ and right-hand side $(\mu_x, \mu_y) \in
\reals^n \times \reals^m$.
There are two main ways we consider to solve this system of equations. The first
way is a \emph{direct} method, which solves the system exactly by initially
computing a sparse permuted $LDL^\top$ factorization of the matrix
\cite{davis2006direct}, caching this factorization, and reusing it every
iteration thereafter. In the majority of cases the factorization cost is greater
than the solve cost using the factors, so once the initial work is done the subsequent
iterations are much cheaper.  Since $P \succeq 0$ this matrix above is
\emph{quasidefinite}, which implies that the $LDL^\top$ factorization exists for
any symmetric permutation \cite{vanderbei1995symmetric}.

Alternatively, we can apply an \emph{indirect} method to solve the system
approximately at each iteration. DR splitting is robust
to inexact evaluations of the resolvent operators and convergence can
still be guaranteed so long as the errors satisfy a summability condition
\cite{eckstein1992douglas}.  To use an indirect method we first reduce this
system by elimination to
\begin{align*}
x &= (I+P + A^TA)^{-1}(\mu_x - A^T \mu_y)\\
y &= \mu_y - Ax,
\end{align*}
and note that the matrix $I + P + A^TA$ is positive definite. This system is then
solved with a conjugate gradient (CG) or similar method
\cite{nocedal2006numerical, ocpb:16}.  One iteration of CG requires multiplications with
the matrices $P$, $A$, and $A^\top$. If these matrices are very sparse, or fast
multiplication routines exist for them, then one CG step can be very fast. We run
CG until the residual satisfies an error bound, at which point we return the
approximate solution.  We can use techniques from the literature, such as
warm-starting CG from the previous solution and using a preconditioner to
improve the convergence \cite{bredies2015preconditioned}.

\subsection{Cone projection}
Most convex optimization problems of interest can be expressed using a
combination of the `standard' cones, namely the positive orthant, second-order
cone, semidefinite cone, and the exponential cone \cite{nesterov1994interior,
nemirovski2007advances}.  These cones all have
well-known projection operators \cite{parikh2014proximal}. Of these, only the
semidefinite cone projection provides a computational challenge since it
requires an eigen-decomposition, which may be costly.  If our problem consists of
the Cartesian product of many of these cones then each of these projections can
be carried out independently and in parallel.

Alternatively, since the cone projection step is totally separated from the rest of the
algorithm, we can incorporate any number of problem-specific cones with
their own projection operators, which may perform better in practice
than reformulating the problem to use the standard cones.
The restriction that the set be a cone is not
too stringent, because we can write many convex constraints as a combination of a
conic constraint and an affine constraint. In particular the set defined by a
convex function $f$ can be transformed as follows
\[
\{s \mid f(s) \leq 0\}
\
\Rightarrow
\
\{ (t, s) \mid t f(s / t) \leq 0, t \geq 0\} \cap \{(t, s) \mid t = 1\}
\]
which is a combination of a convex cone and
an affine equality constraint, which fits our framework. If the original convex
set has an efficient projection operation, then in the worst-case we can perform
a bisection search over $t \geq 0$ using the projection operator as a
subroutine.  In most cases the dominant cost of Algorithm \ref{a-drmcp} will be
solving the linear system, so the additional cost of a bisection to compute the
cone projection will typically be negligible. As an example, consider the `box'
cone defined as
\[
\Kc_{\mathrm{box}} = \{ (t, s) \mid t l \leq s \leq t u, t \geq 0\}
\]
where $l, u \in \reals^d$ are data. When combined with the constraint that
$t=1$ this represents box constraints on the variable $s$, which is commonly used
in LP and QP solvers. Projection onto this cone can be done via Newton's method
on the scalar variable $t$, which typically only requires a few iterations to reach
convergence.  This cone is supported in the SCS v3.0 solver.

\section{Numerical experiments}
\subsection{Comparing Algorithm \ref{a-drmcp} to Equation \eqref{e-drlcp}}

Here we compare the computational efficiency of using DR splitting applied to
the homogeneous embedding (Algorithm \ref{a-drmcp}) and DR splitting applied
directly to the original problem (Equation \eqref{e-drlcp}) on a range of
synthetic problems.  We constructed feasible, primal infeasible, and unbounded
(dual infeasible) QCPs over the positive orthant and compared the
number of iterations taken by Equation \eqref{e-drlcp} with infeasibility
detection using successive iterates and Algorithm \ref{a-drmcp}. Since the cost
per iteration is essentially identical for both approaches the number of
iterations determines the overall solve time.  The results on diverging
sequences producing infeasibility certificates from Banjac et
al.~\cite{banjac2019infeasibility}, and Liu et al.~\cite{liu2017new} do not
immediately carry over to the case of Equation \eqref{e-drlcp} since they only
hold for ADMM applied to convex functions, and the matrix $M$ is not the
subdifferential of a convex function.  That being said, we can still use the
techniques and compare the performance in practice. In the sequel we shall
compare solvers that do come with theoretical guarantees.

We randomly generated $1000$ feasible, infeasible, and unbounded problems of
size $n=100$ and $m=150$. For feasible problems we declared the problem to be
solved when the maximum $\ell_\infty$-norm KKT violation was $10^{-6}$.
Similarly, for infeasible and unbounded problems we stopped when the algorithms
produces a valid certificate with $\ell_\infty$-norm tolerance of $10^{-6}$.
For each problem we computed the ratio of the number of iterations required by
Equation \eqref{e-drlcp} to the number required by Algorithm \ref{a-drmcp} to
solve the problem or certify infeasibility. A higher ratio indicates that
Algorithm \ref{a-drmcp} requires fewer iterations to solve the problem than
Equation \eqref{e-drlcp}.   We present histograms of the performance ratio in
Figures~\ref{f-feas-qp}, \ref{f-infeas-qp}, and~\ref{f-unbdd-qp} for feasible,
infeasible, and unbounded problems respectively.  Evidently, generating
certificates from the homogeneous embedding can be orders of magnitude faster;
the geometric mean of the ratio on infeasible problems was $49.0$ and on
unbounded problems was $299.1$.  In fact our approach was not slower on a single
instance.  The successive differences approach failed to find a certificate of
infeasibility within the iteration limit of $10^5$ in $27$ problems.  For
feasible problems the approach based on the homogeneous embedding is often
quicker to find a solution, sometimes by a significant factor.  The geometric
mean of the ratios was $1.6$, and the homogeneous embedding approach was faster
in $987$ of the $1000$ problems.

In Figure \ref{f-traces} we show how the maximum $\ell_\infty$-norm residuals
converge on randomly selected feasible, infeasible, and unbounded problems.
For the feasible problem we plot the maximum KKT condition residual and for the
infeasible and unbounded problems we plot the maximum residual from a valid
certificate.  For infeasible and unbounded problems the approach
based on the homogeneous embedding converges to a certificate extremely rapidly,
but the approach based on diverging iterates takes many iterations to produce
a certificate.  For the feasible problem the difference is less stark, but
Algorithm \ref{e-drmcp} still converges faster, reaching the tolerance
in about half the number of iterations required by Equation \eqref{e-drlcp}.

\begin{figure}[h]
\centering
\begin{subfigure}{0.32\textwidth}
\centering
\includegraphics[width=\linewidth]{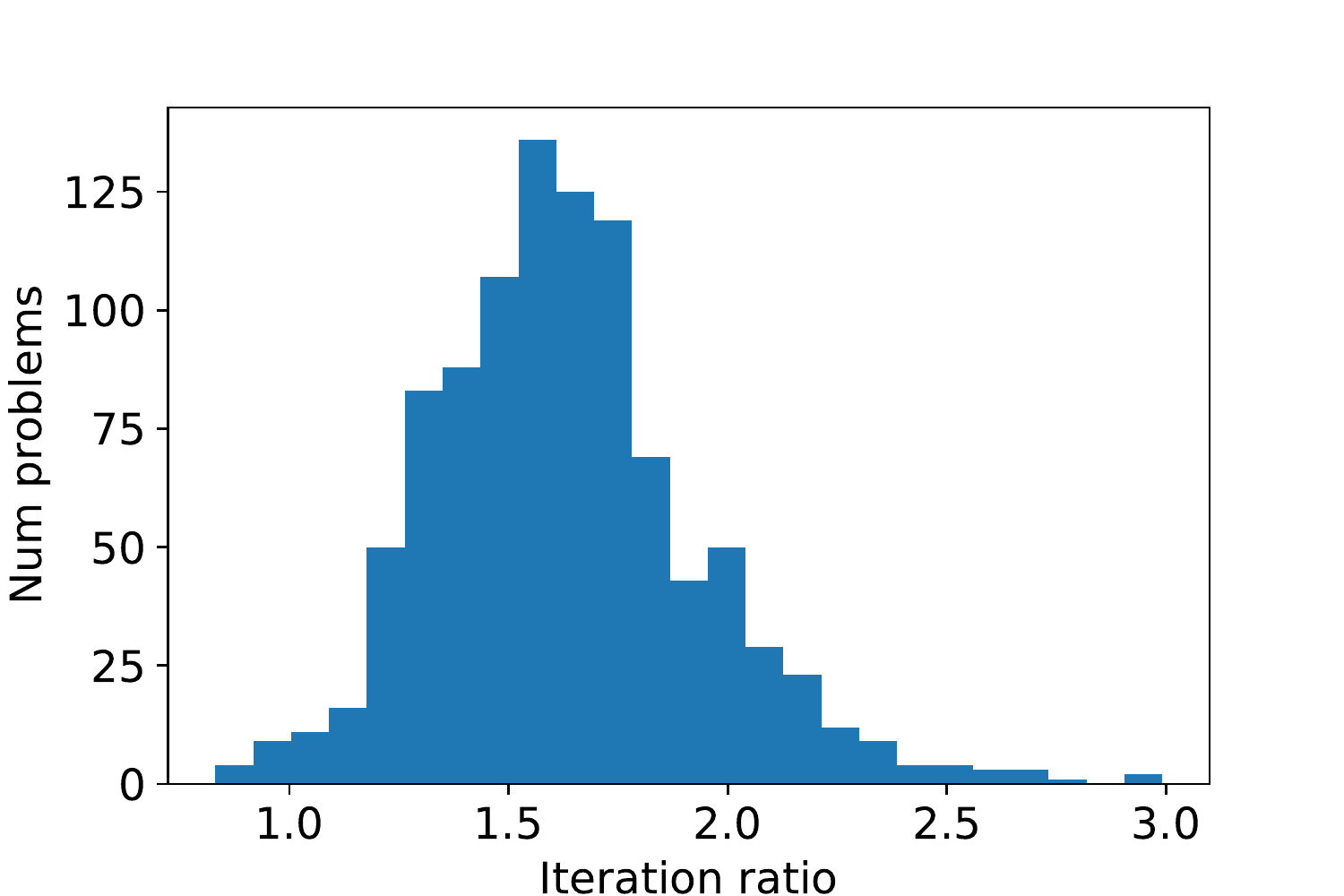}
\caption{Feasible QCPs.}
\label{f-feas-qp}
\end{subfigure}
\begin{subfigure}{0.32\textwidth}
\centering
\includegraphics[width=\linewidth]{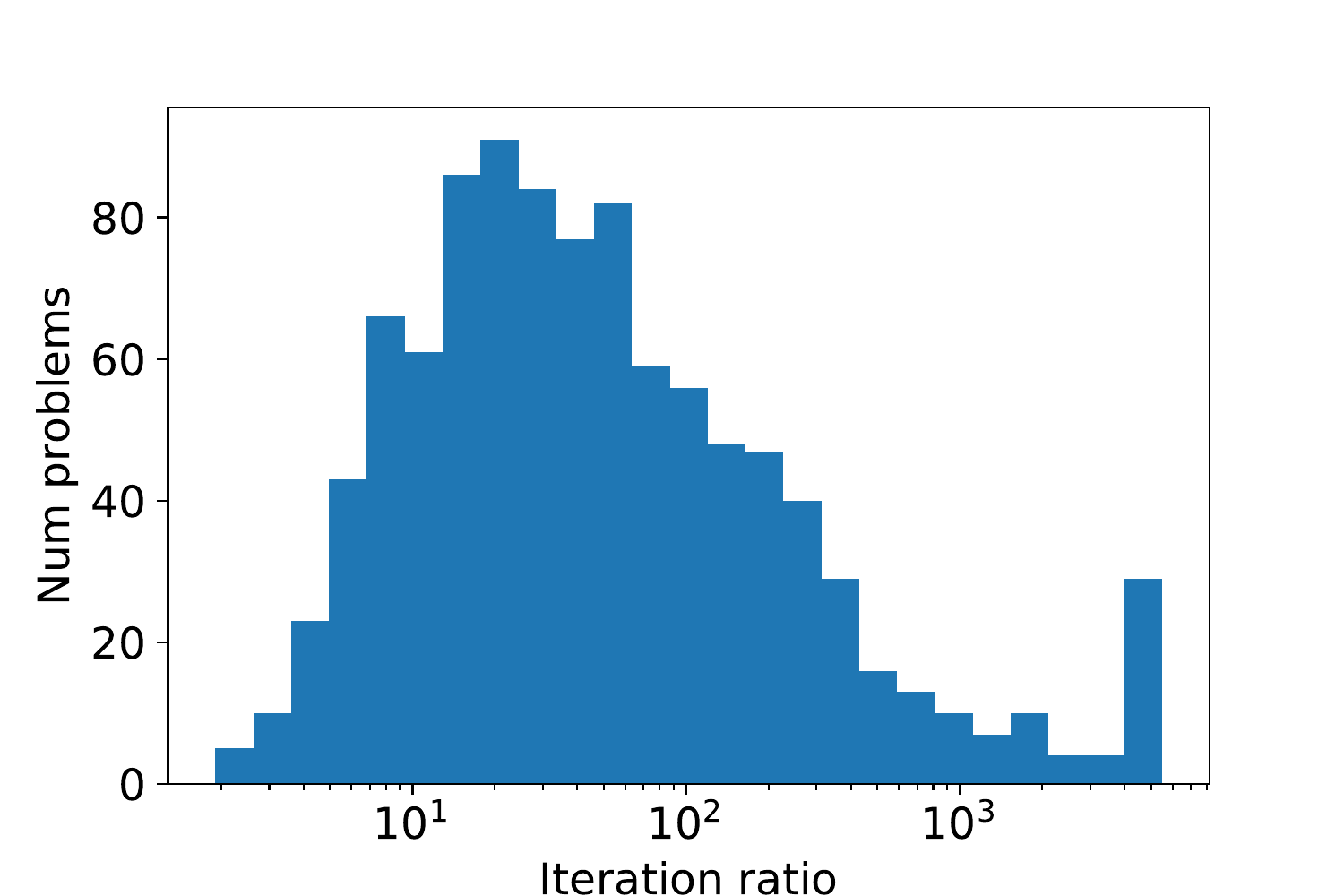}
\caption{Infeasible QCPs.}
\label{f-infeas-qp}
\end{subfigure}
\begin{subfigure}{0.32\textwidth}
\centering
\includegraphics[width=\linewidth]{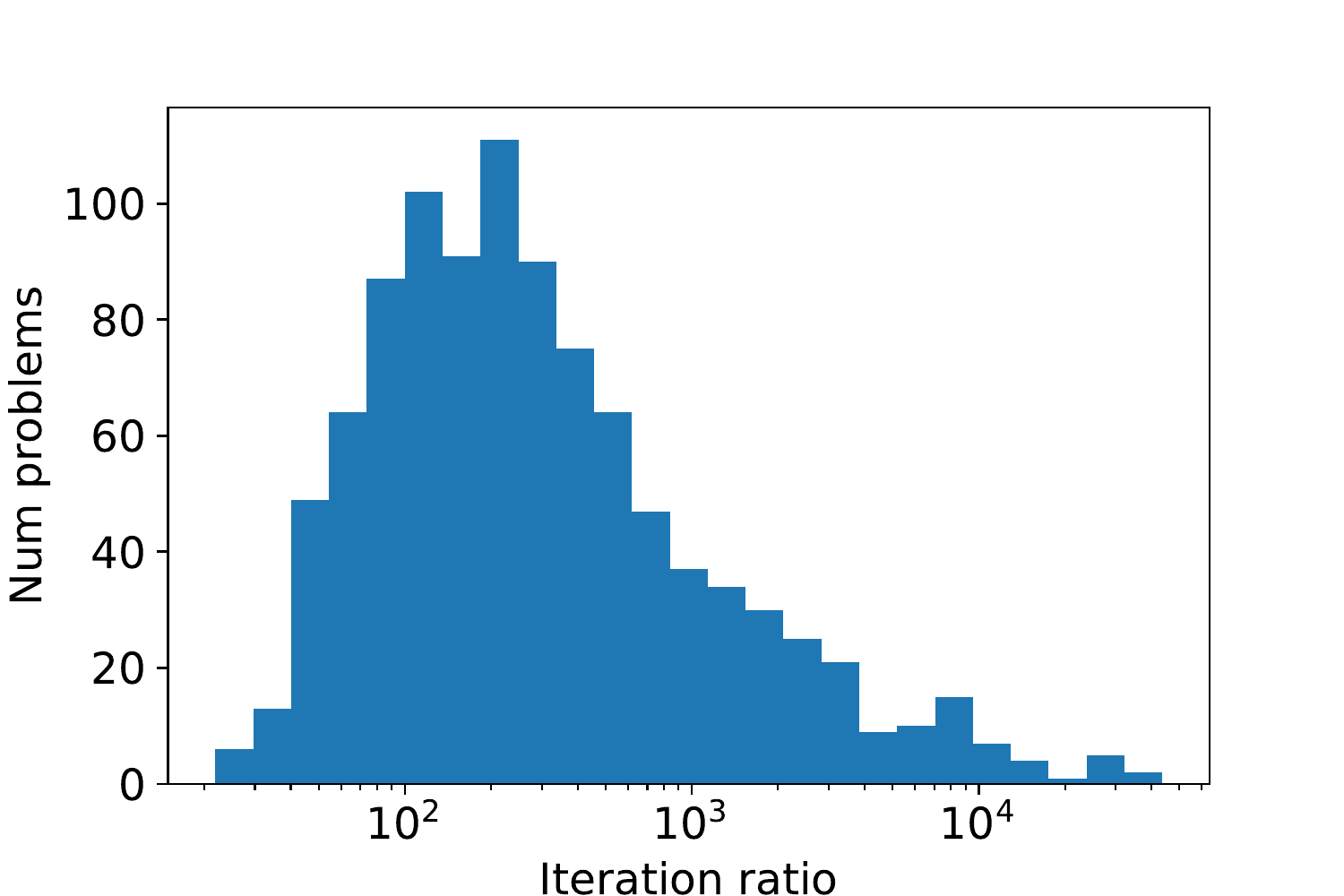}
\caption{Unbounded QCPs.}
\label{f-unbdd-qp}
\end{subfigure}
\caption{Histograms of iteration count ratio of Equation~\eqref{e-drlcp}
to Algorithm~\ref{a-drmcp}.  Higher ratios indicates that our approach is taking
fewer iterations to reach the same accuracy.}
\label{f-infeas-unbdd-qp}
\end{figure}

\begin{figure}[h]
\centering
\begin{subfigure}{0.32\textwidth}
\centering
\includegraphics[width=\linewidth]{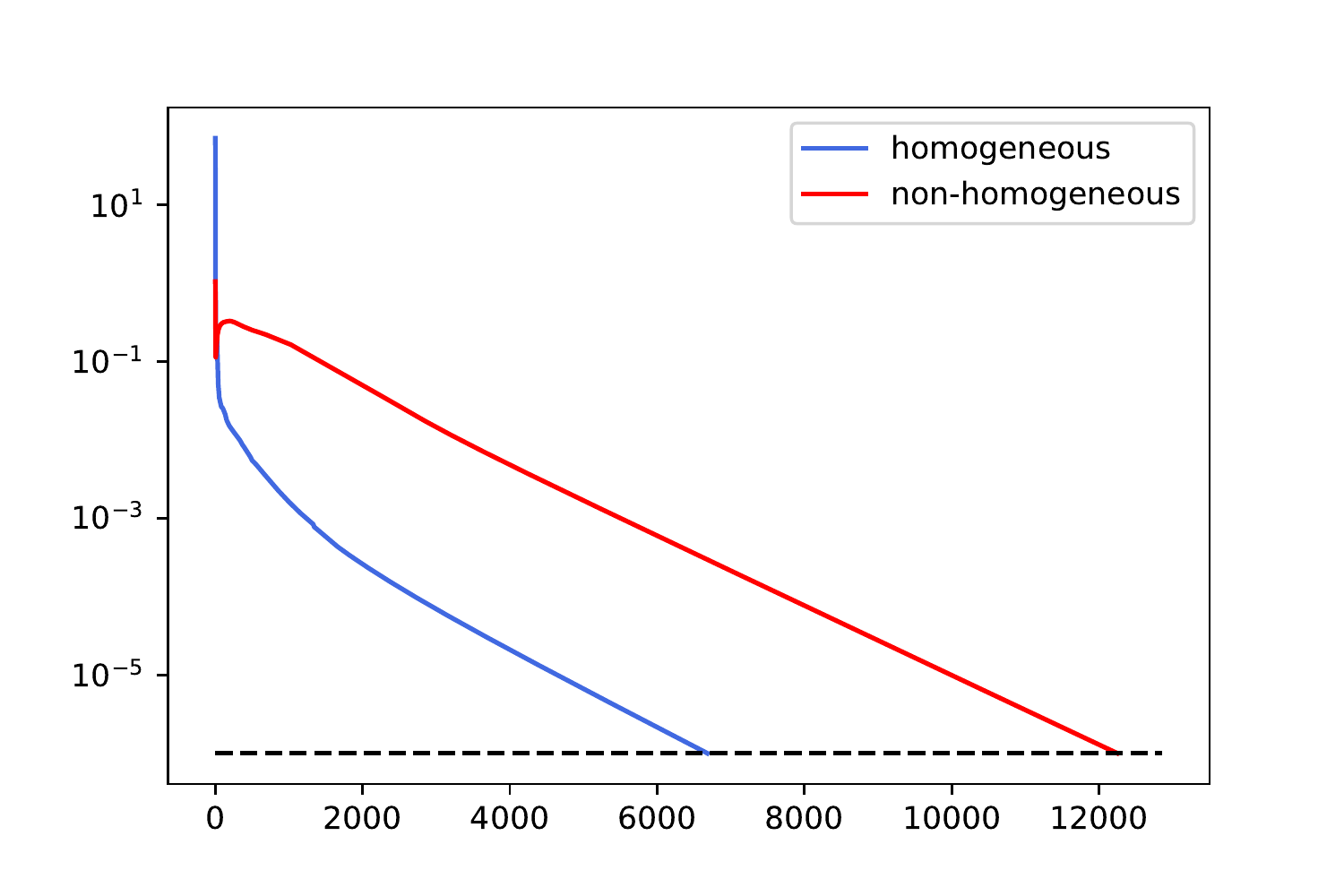}
\caption{Feasible QCP.}
\end{subfigure}
\begin{subfigure}{0.32\textwidth}
\centering
\includegraphics[width=\linewidth]{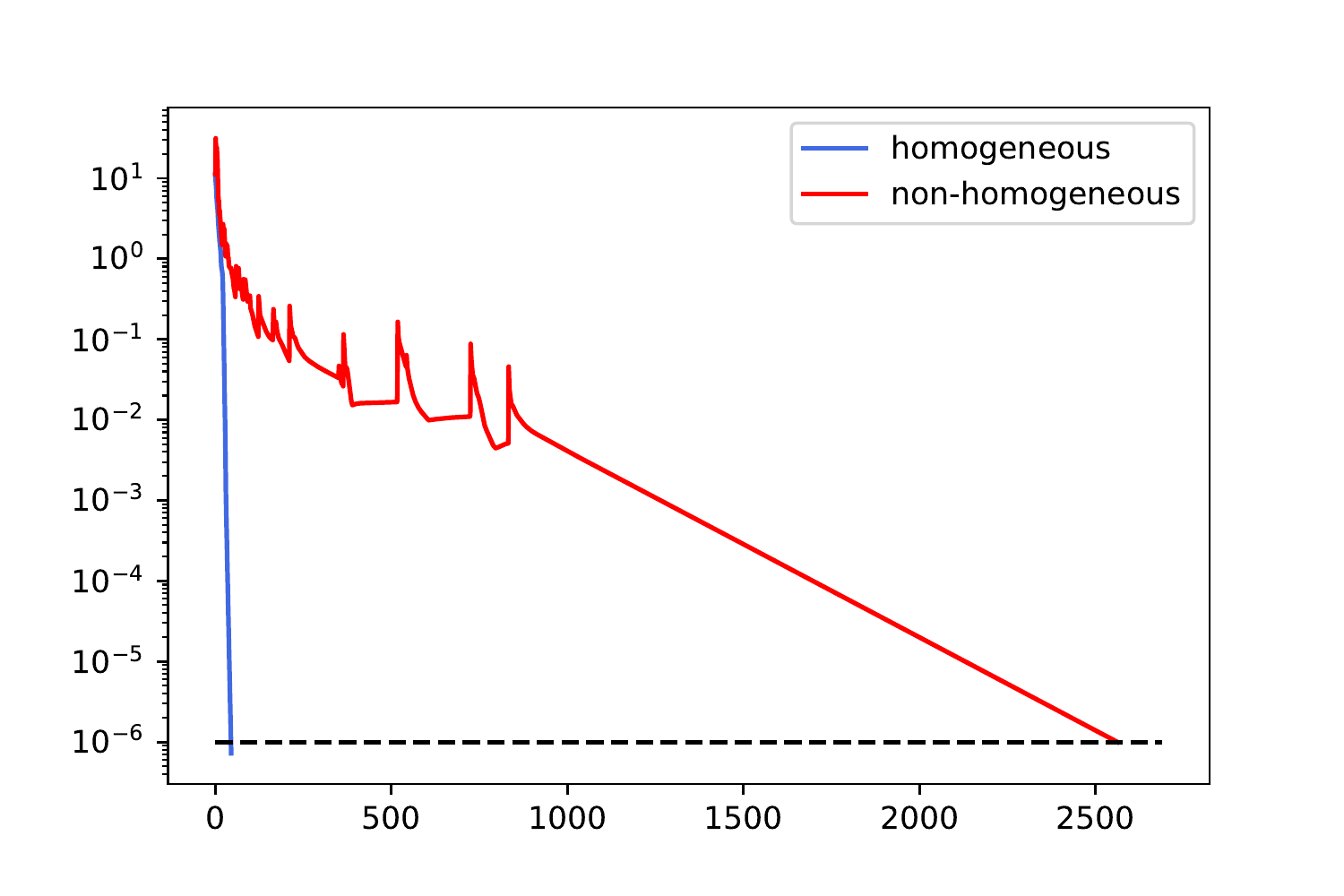}
\caption{Infeasible QCP.}
\end{subfigure}
\begin{subfigure}{0.32\textwidth}
\centering
\includegraphics[width=\linewidth]{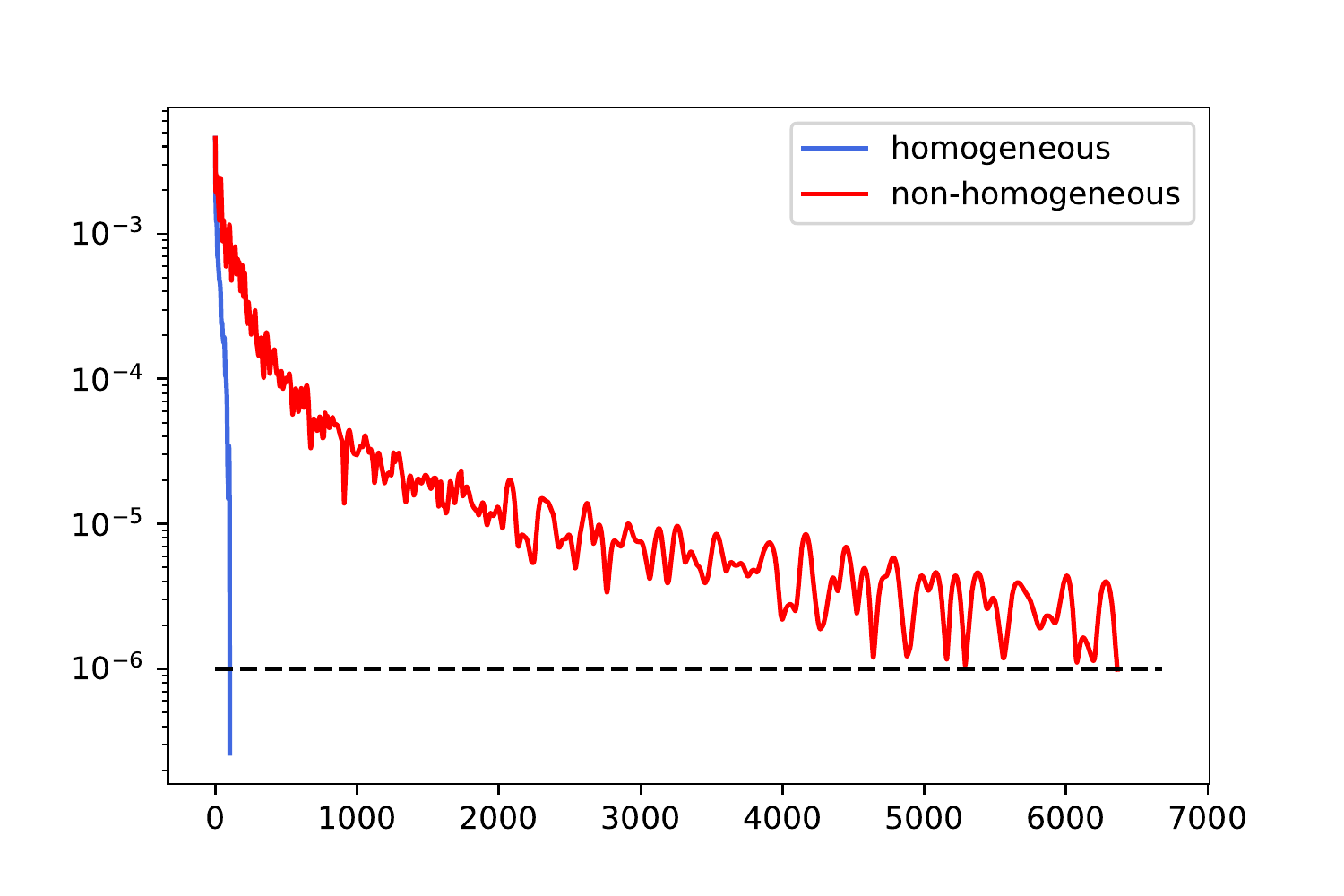}
\caption{Unbounded QCP.}
\end{subfigure}
\caption{Trace of max residuals under Equation~\eqref{e-drlcp} (non-homogeneous)
and Algorithm~\ref{a-drmcp} (homogeneous) for randomly selected example problems.}
\label{f-traces}
\end{figure}

\subsection{Comparing open-source solvers}

In this section we compare SCS v3.0, our open-source implementation of
Algorithm \ref{a-drmcp} for QCPs, to other available open-source solvers that
apply ADMM directly to QCPs. In particular we compare to OSQP
\cite{stellato2018osqp} and COSMO \cite{garstka_2019} both of which rely on
diverging iterates to generate certificates of infeasibility.

As discussed in \S \ref{s-convergence} the iterates produced by SCS v3.0 always
satisfy the cone membership and complementarity KKT conditions defined in Equation \eqref{e-kkt}.
Therefore to say that a problem is solved we need to check if the primal residual,
dual residual, and duality gap are all below a certain tolerance. Specifically,
SCS v3.0 terminates when it has found $x \in \reals^n$, $ s \in \reals^m$,
and $y \in \reals^m$ that satisfy
\begin{align*}
\|Ax + s - b\|_\infty &\leq \epsilon_\mathrm{abs} + \epsilon_\mathrm{rel} \max(\|Ax\|_\infty, \|s\|_\infty, \|b\|_\infty) \\
\|Px + A^\top y - c \|_\infty &\leq \epsilon_\mathrm{abs} +
\epsilon_\mathrm{rel} \max(\|Px\|_\infty, \|A^\top y\|_\infty, \|c\|_\infty)\\
|x^\top Px + c^\top x + b^\top y| &\leq \epsilon_\mathrm{abs}
+ \epsilon_\mathrm{rel} \max(|x^\top P x|, |c^\top x|, |b^\top y|),
\end{align*}
where $\epsilon_\mathrm{abs}>0$ and $\epsilon_\mathrm{rel}>0$ are user defined
quantities that control the accuracy of the solution. For the purposes of
our experimental results we set $\epsilon_\mathrm{abs}=10^{-3}$ and
$\epsilon_\mathrm{rel}=10^{-4}$.  OSQP and COSMO have analogous quantities for
the primal and dual residual, however, they do not allow the user to specify a bound
on the gap. Therefore, in order to ensure that the gap is below the desired
tolerance we solve each problem with these solvers with the initial choices of
$\epsilon_\mathrm{abs}$ and $\epsilon_\mathrm{rel}$ and check if the gap is
below the tolerance. If it is then we return that solution, otherwise we halve
$\epsilon_\mathrm{abs}$ and $\epsilon_\mathrm{rel}$ and re-solve. This procedure
is continued until the solver returns a solution that satisfies the gap
constraint, and only the last solve counts towards the statistics. For a
concrete case of why this is necessary take the \verb|BOYD2| problem from the
Maros-M{\'e}sz{\'a}ros QP dataset. OSQP returns the certificate `solved' for
this problem after $24300$ iterations with an `optimal' objective of $343.32$.
However, the true optimal objective value for this problem is $21.26$
\cite{maros1999repository}. The issue is that the duality gap of the primal-dual
pair returned by OSQP is $1.3\times 10^3$, when the desired gap is on the order
of $10^{-2}$.  Since the primal and dual residuals are small but the duality gap
is large it means that OSQP has returned a primal-dual pair that is (almost)
feasible, but is far from optimal.  On the other hand SCS v3.0, which only
terminates when the gap as well as the primal and dual residuals are below the
tolerance, returns a solution after $3250$ iterations with an objective value of
$21.12$, significantly closer to the true value.

Since the cone memberships are always guaranteed by the iterates, SCS v3.0
declares a problem infeasible when it finds $y \in \reals^m$ that satisfies
\[
\begin{array}{lr}
b^\top y = -1, & \|A^\top y\|_\infty < \epsilon_\mathrm{infeas}.
\end{array}
\]
Similarly, SCS v3.0 declares dual infeasibility when it finds $x \in
\reals^n$, $s \in \reals^m$ that satisfy
\[
\begin{array}{lr}
c^\top x = -1, & \max(\|P x\|_\infty, \|A x + s\|_\infty) <
\epsilon_\mathrm{infeas}.
\end{array}
\]
The other solvers have analogous certificates, and in these cases there is no
duality gap so the iterative procedure is not required.
For the experiments we set $\epsilon_\mathrm{infeas} = 10^{-4}$.

All three solvers rescale the data to yield better conditioning and they all
implement a heuristic `step-size' adaptation scheme.  These heuristics were
enabled for these experiments, however we note that the conclusions we derive
from the experiments did not change when these heuristics were disabled. On the
contrary, the advantage that the homogeneous embedding had over the direct
approaches was more pronounced in that case.  We disabled more advanced
techniques like acceleration, solution polishing, and semidefinite cone
decomposition.  All three solvers were given a maximum iteration limit of $10^5$
and a time-limit of $10^3$ seconds per problem. If a solver fails to find a
solution or a certificate of infeasibility satisfying the tolerances within
those limits then it is considered to have failed to solve that problem. When
measuring average run-times any failures are assigned the maximum run-time of
$10^3$ seconds. All experiments were run single-threaded on a 2017 MacBook pro
with a $3.1$Ghz Intel i7 and $16$Gb of RAM.

We present results on several datasets. First we present results on the
Maros-M{\'e}sz{\'a}ros dataset of challenging convex feasible QPs
\cite{maros1999repository}.  Next, the NETLIB dataset, which contains both
feasible and infeasible linear programs \cite{netlib}. The SDPLIB dataset also
has 4 infeasible problems, on which we test SCS v3.0 and COSMO \cite{sdplib}
(OSQP does not support the semidefinite cone).
Finally, we present results on randomly generated quadratic problems as in the
previous section.
To summarize the results for each dataset we shall use Dolan-Mor{\'e}
performance profiles \cite{dolan2002benchmarking}. In these plots each point of
the curve corresponds to what fraction of the problems are solved ($y$-axis)
within a factor ($x$-axis) of the fastest solver for each problem. Curves of
faster solvers appear above those of slower solvers.
When summarizing wall-clock performance we shall use the shifted geometric means
of the run-times with a shift of $10$ seconds, denoted \verb|sgm10|.

In Figure~\ref{f-maros-meszaros} we show the Dolan-Mor{\'e} profile for the
Maros-M{\'e}sz{\'a}ros QP dataset and in Table \ref{t-maros meszaros
problems-failure} we present the failure rates. SCS v3.0 is the most robust
solver, with around a third of the failures of the next best solver. In terms of
solve speeds SCS v3.0 was the fastest, followed by COSMO which was about
$2.6\times$ slower and then OSQP which was about $2.8\times$ slower, as measured
by \verb|sgm10|.

In Figure~\ref{f-netlib} we show the profiles for the NETLIB dataset, broken
down into feasible and infeasible problems. In this case it is clear that
SCS v3.0 is the fastest solver. This is partially explained by the fact that
SCS v3.0 appears to be far more robust for these problems with a
significantly lower overall failure rate, as shown in Tables \ref{t-netlib
infeasible-failure} and \ref{t-netlib feasible-failure}.  For the feasible
problems SCS v3.0 was about $16\times$ faster than OSQP and $20\times$ faster than
COSMO as measured by \verb|sgm10|.  For the infeasible problems SCS v3.0 was
about $3.2\times$ faster than OSQP and $20\times$ faster than COSMO.

The results for all four infeasible SDPLIB problems are given in table
\ref{t-sdplib infeasible-results}. Both SCS v3.0 and COSMO successfully
certified that these problems were infeasible (primal or dual depending on the
problem), but SCS v3.0 was able to certify infeasibility significantly faster
than COSMO, about $66\times$ faster in terms of \verb|sgm10|.  This difference
is partially explained by the number of iterations required to generate a
certificate. COSMO required almost $10\times$ the number of iterations of SCS
v3.0 to certify that these problems were infeasible.

Finally, the results for the random QPs are presented in Figure \ref{f-random}.
For feasible random problems SCS v3.0 and OSQP have similar performance, with
OSQP about $4\%$ faster than SCS v3.0 on average, and COSMO somewhat slower.
All three solvers solved all feasible instances.  However, for
the randomly generated infeasible and unbounded problems the difference is
stark.  For unbounded problems SCS v3.0 certified every single problem correctly,
OSQP had a $0.8\%$ failure rate and COSMO had a $1.2\%$ failure rate.  However,
SCS v3.0 was about $32\times$ faster than OSQP and $41\times$ faster than COSMO,
as measured by \verb|sgm10|.  For the infeasible problems again SCS v3.0 was able
to certify infeasibility on all problems, OSQP on all but one problem, but COSMO
was unable to certify infeasibility on even a single instance, hitting the
maximum iteration limit on every problem.  Even when the infeasibility
tolerances were loosened COSMO still struggled. This explains the strange
Dolan-Mor{\'e} profile for this problem set, where SCS v3.0 is barely visible at
the top left and COSMO barely visible in the bottom right. Even though OSQP and
SCS v3.0 had similar success rates, SCS v3.0 was able to certify infeasibility
about $29\times$ faster as measured by \verb|sgm10|.

\begin{figure}[h]
\centering
\includegraphics[width=0.5\linewidth]{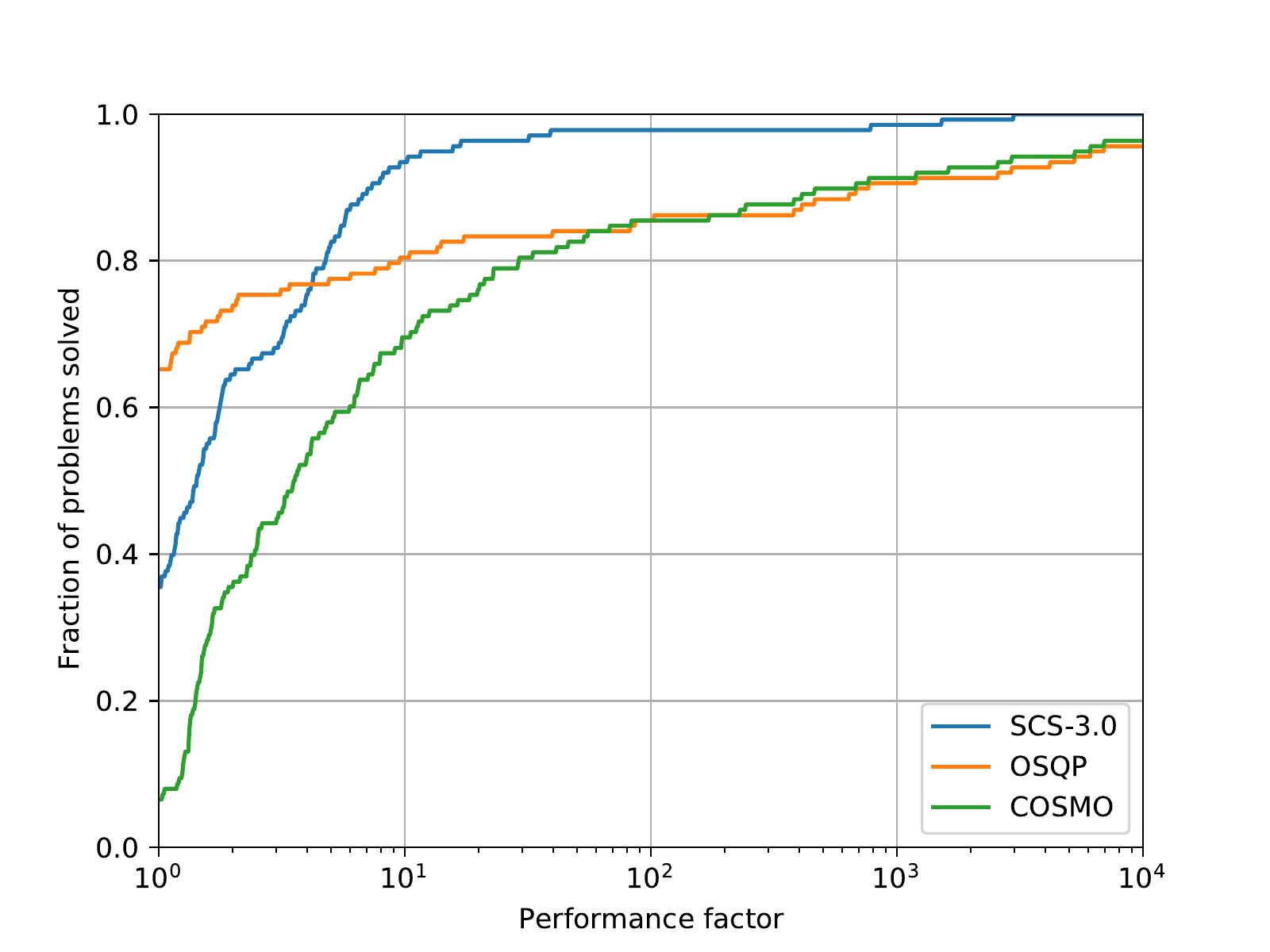}
\caption{Performance profiles for Maros-M{\'e}sz{\'a}ros problems.}
\label{f-maros-meszaros}
\end{figure}

\begin{table}
\centering
\caption{Solver failure rates on Maros-M{\'e}sz{\'a}ros problems.}
\label{t-maros meszaros problems-failure}
\begin{tabular}{lccc}
\toprule
 SCS-3.0 &  OSQP &  COSMO \\
\midrule
5.80\% & 18.12\% &  16.67\% \\
\bottomrule
\end{tabular}
\end{table}

\begin{figure}[h]
\centering
\begin{subfigure}{0.49\textwidth}
\centering
\includegraphics[width=\linewidth]{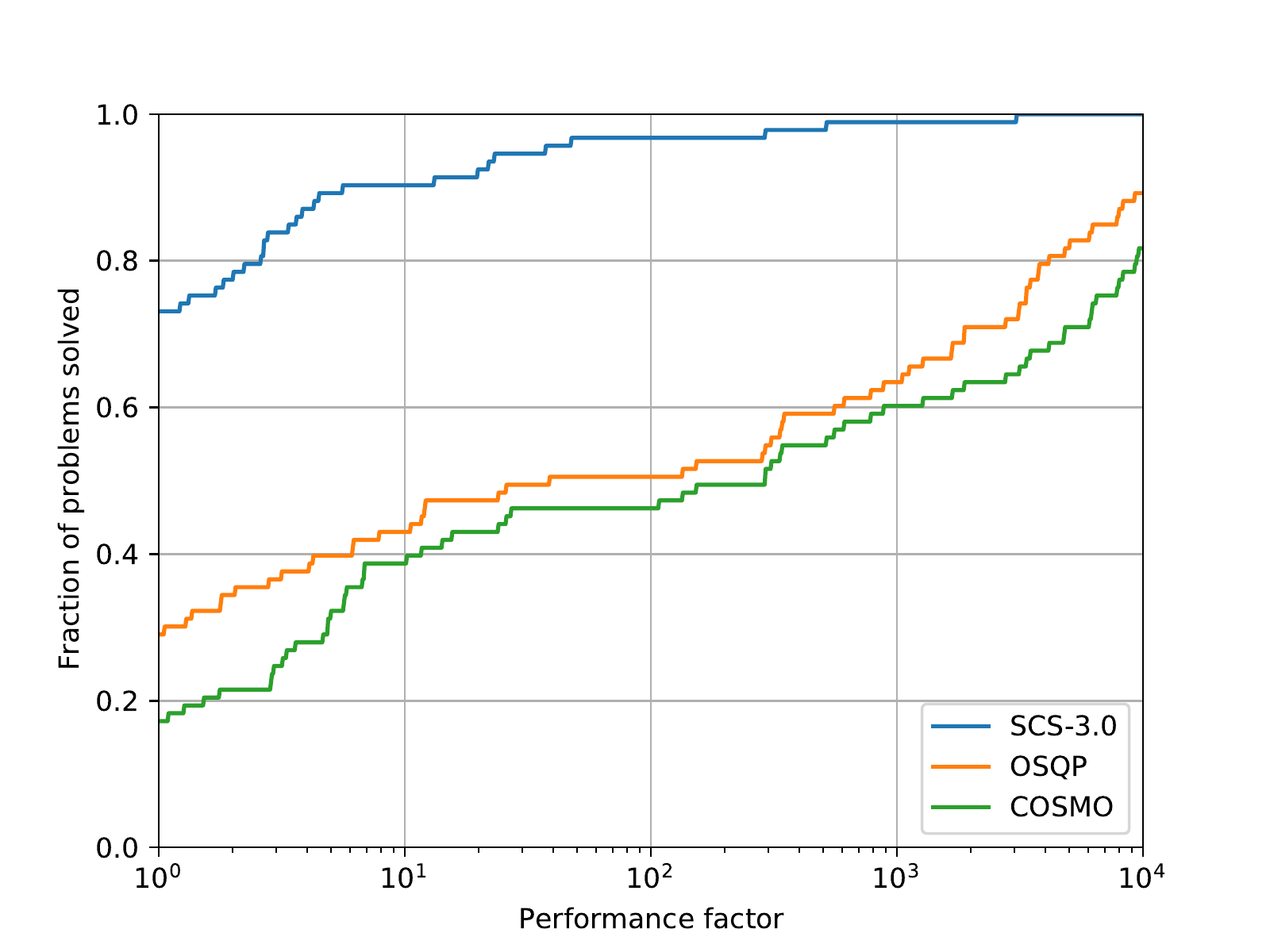}
\caption{Feasible problems.}
\label{f-netlib-feasible}
\end{subfigure}
\begin{subfigure}{0.49\textwidth}
\centering
\includegraphics[width=\linewidth]{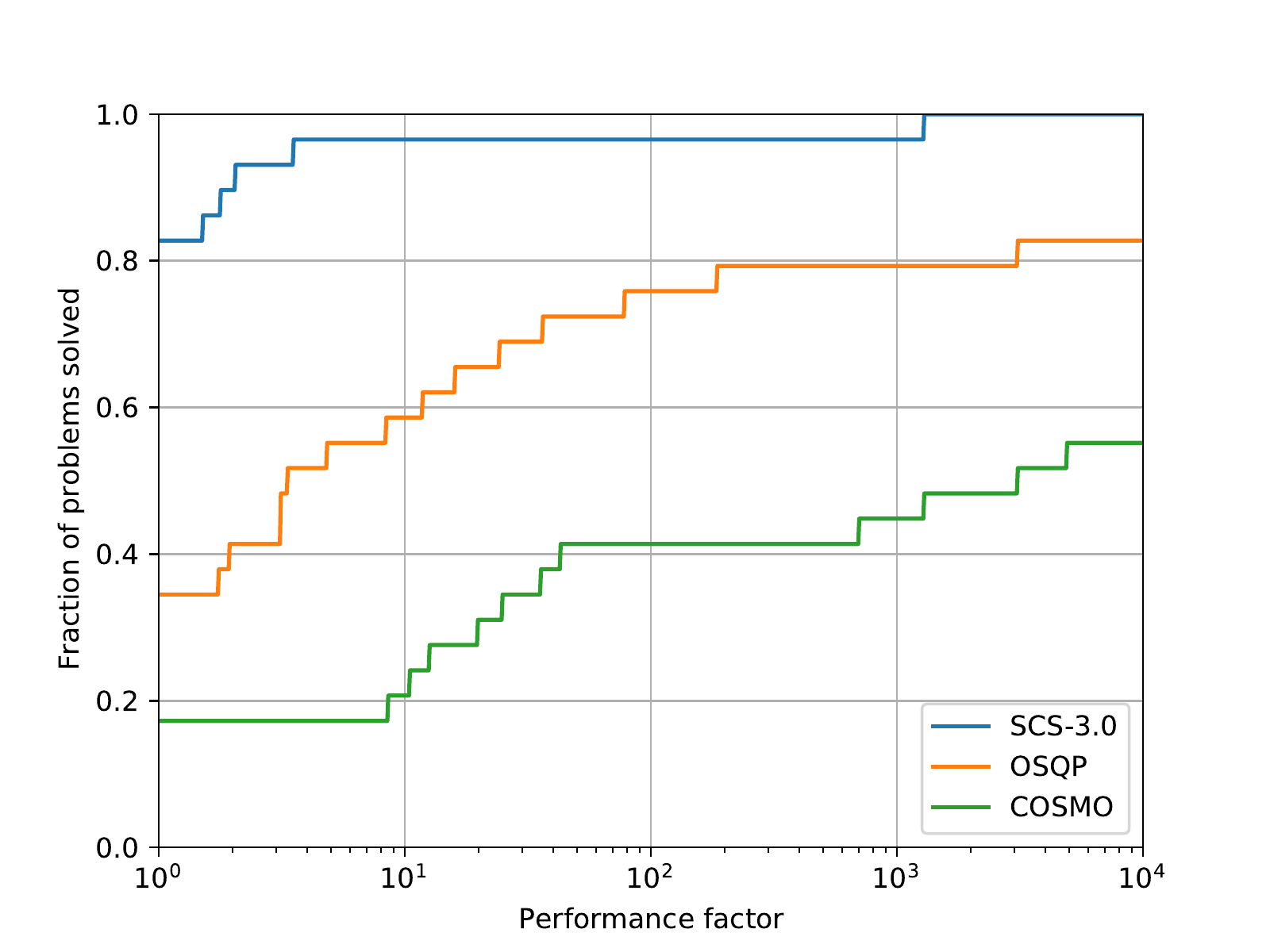}
\caption{Infeasible problems.}
\label{f-netlib-infeasible}
\end{subfigure}
\caption{Performance profiles for NETLIB LP problems.}
\label{f-netlib}
\end{figure}
\begin{table}
\centering
\caption{Solver failure rates on NETLIB infeasible problems.}
\label{t-netlib infeasible-failure}
\begin{tabular}{lccc}
\toprule
  SCS-3.0 &  OSQP &  COSMO \\
\midrule
20.69 \% & 37.93\% &  75.86\% \\
\bottomrule
\end{tabular}
\end{table}

\begin{table}
\centering
\caption{Solver failure rates on NETLIB feasible problems.}
\label{t-netlib feasible-failure}
\begin{tabular}{lccc}
\toprule
  SCS-3.0 &  OSQP &  COSMO \\
\midrule
12.90 \% & 61.29\% &  65.59\% \\
\bottomrule
\end{tabular}
\end{table}

\begin{table}
\centering
\caption{Solver times on SDPLIB infeasible problems in seconds.}
\label{t-sdplib infeasible-results}
\begin{tabular}{lcc}
\toprule
{} &  SCS-3.0 &  COSMO \\
\midrule
infd1 & 0.0122 & 2.1776 \\
infd2 & 0.0155 & 0.0321 \\
infp1 & 0.0035 & 0.1154 \\
infp2 & 0.0037 & 0.1119 \\
\bottomrule
\end{tabular}
\end{table}

\begin{figure}[h]
\centering
\begin{subfigure}{0.32\textwidth}
\centering
\includegraphics[width=\linewidth]{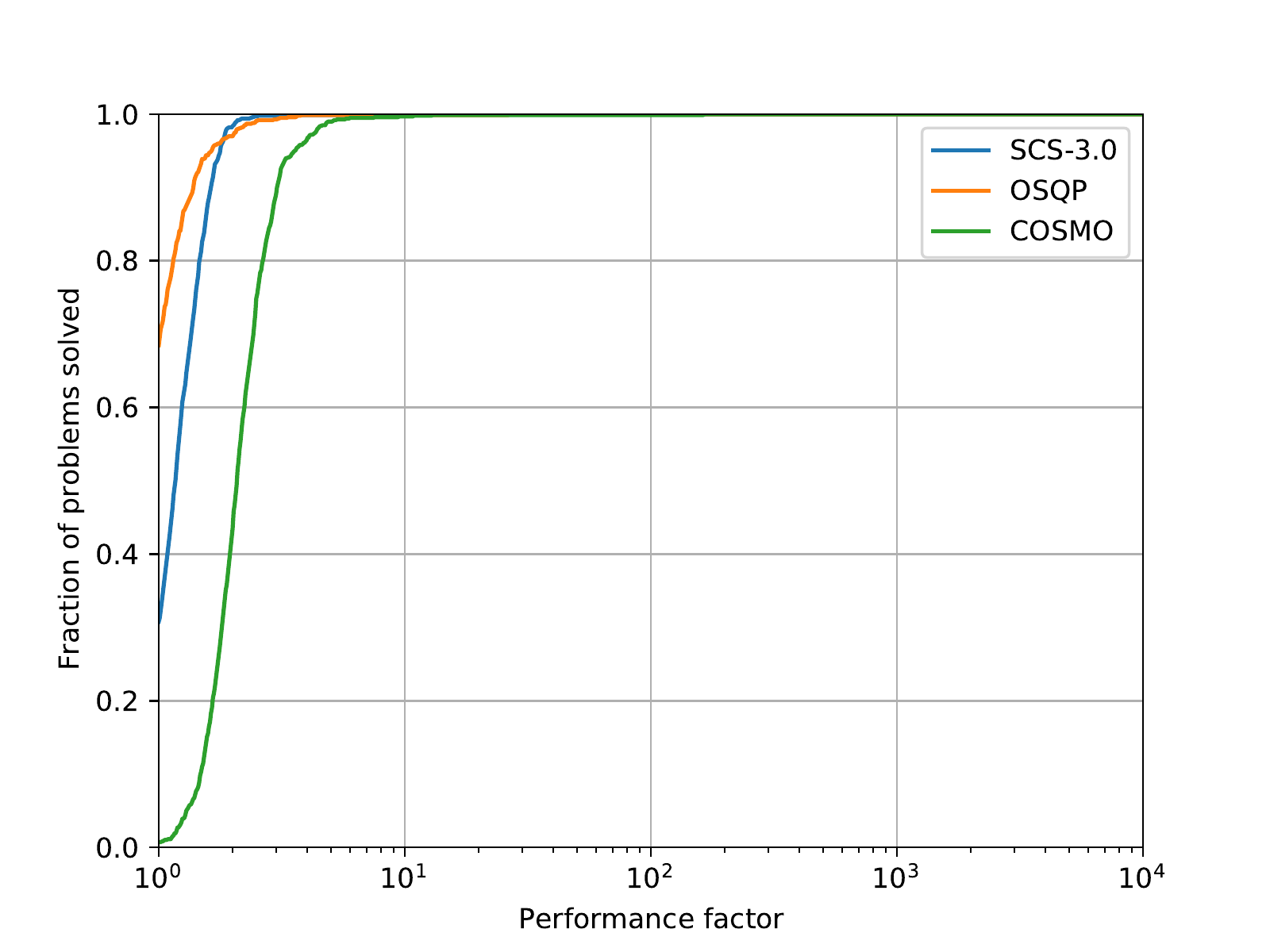}
\caption{Feasible problems.}
\label{f-random-feasible}
\end{subfigure}
\begin{subfigure}{0.32\textwidth}
\centering
\includegraphics[width=\linewidth]{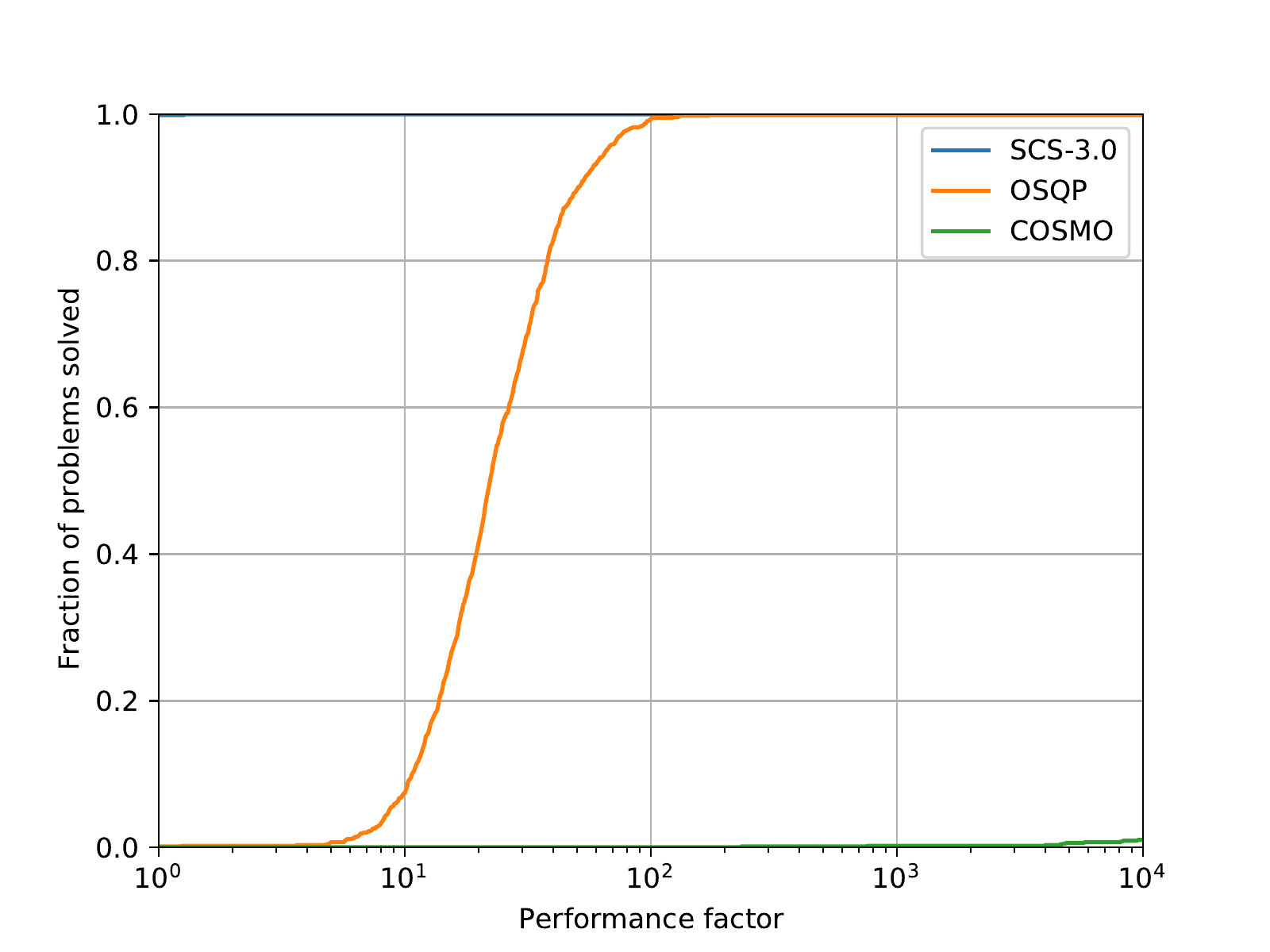}
\caption{Infeasible problems.}
\label{f-random-infeasible}
\end{subfigure}
\begin{subfigure}{0.32\textwidth}
\centering
\includegraphics[width=\linewidth]{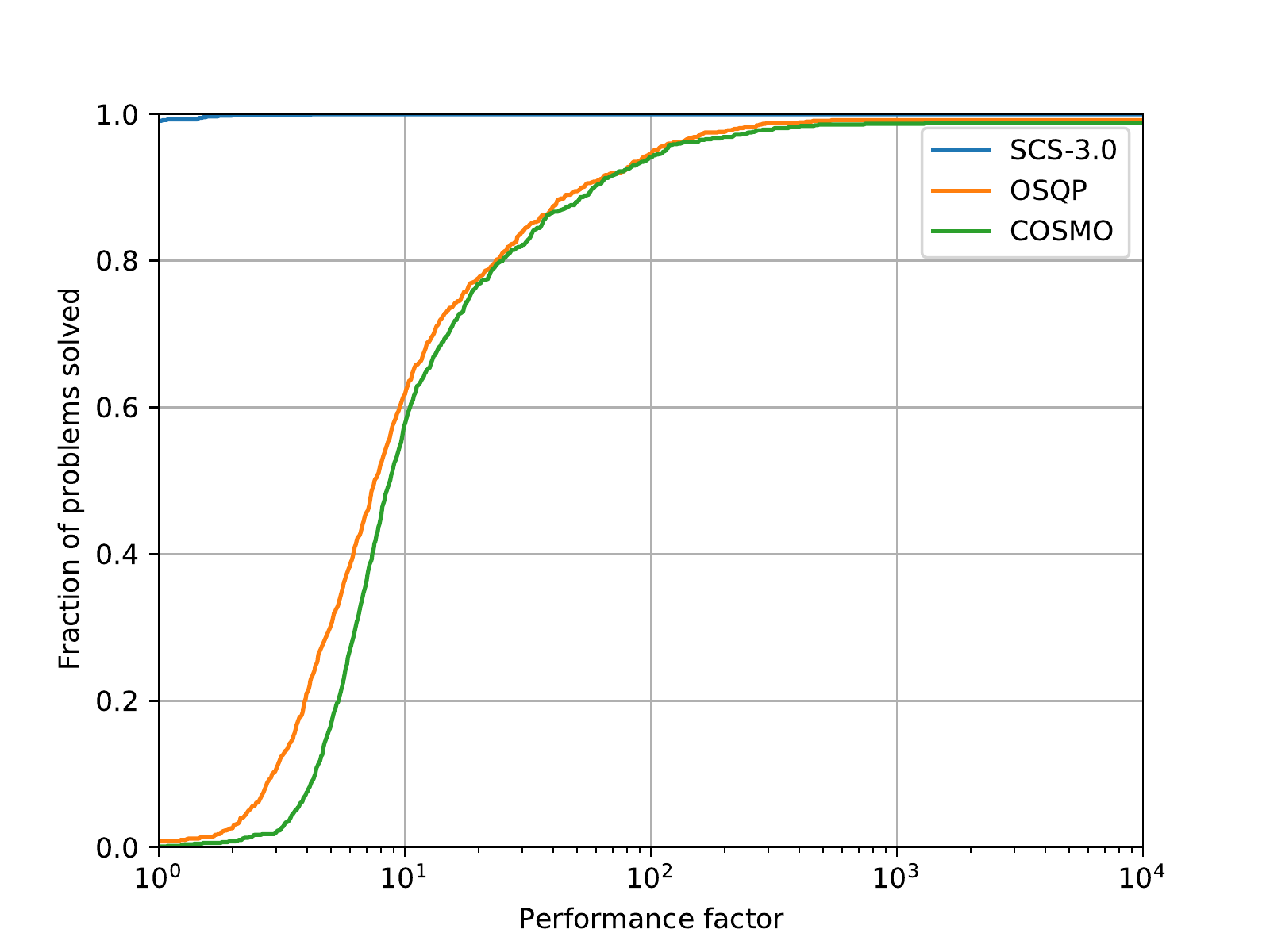}
\caption{Unbounded problems.}
\label{f-random-unbounded}
\end{subfigure}
\caption{Performance profiles for randomly generated QP problems.}
\label{f-random}
\end{figure}

\section{Conclusion}
We applied Douglas-Rachford splitting to a homogeneous embedding of the linear
complementarity problem (LCP). This resulted in a simple alternating procedure
in which we solve a linear system and project onto a cone at each iteration.
Since the linear system does not change from one iteration to the next we can
factorize the matrix once and cache it for use thereafter.  Our procedure is
able to return the solution to the LCP when one exists, or a certificate of
infeasibility otherwise. Quadratic cone programs (QCP) are an important special
case of LCPs and we discussed how to implement the procedure efficiently for
QCPs in detail.  We concluded with some experiments demonstrating the
advantage of our procedure over competing approaches numerically, showing large
speedups for infeasible problems without sacrificing performance on feasible
problems. The algorithm has been implemented in C and is available as an
open-source QCP solver.

\section*{Acknowledgments}
The author would like to thank his friends and colleagues at DeepMind for their
support and encouragement. He is also deeply indebted to three anonymous
referees for their careful reviews and excellent feedback.

\bibliographystyle{siamplain}
\bibliography{homogeneous_vi}

\end{document}